\documentclass[final, nomarks]{dmtcs-episciences}

\usepackage[utf8]{inputenc}
\usepackage[T1]{fontenc}
\usepackage[english]{babel}
\usepackage[margin=1.1in]{geometry}

\usepackage{adjustbox}
\usepackage{amsmath}
\usepackage{amssymb}
\usepackage{amsthm}
\usepackage{comment}
\usepackage{enumitem}
\usepackage{epstopdf}
\usepackage{floatpag}
\floatpagestyle{empty}
\usepackage{graphicx}
\usepackage{subfigure}
\usepackage{tikz}
\usetikzlibrary{arrows}
\usepackage{xspace}
\usepackage{xurl}

\usepackage{hyperref}

\usepackage{coge-dmtcs}
\usepackage[round]{natbib}

\newcommand\disjunion{\mathbin{\dot\cup}}

\newlist{seclist}{enumerate}{1}
\setlist[seclist,1]{label=\S\arabic*,
                   ref=\S\arabic*,
                   leftmargin=*}

\title{Orientation of good covers}
\author{P\'eter \'Agoston\thanks{Supported by the Ministry of Innovation and Technology NRDI Office within the framework of the Artificial Intelligence National Laboratory (RRF-2.3.1-21-2022-00004), by the Lend\"ulet program of the Hungarian Academy of Sciences (MTA), under the grant LP2017-19/2017 and by the Thematic Excellence Program TKP2021-NKTA-62 of the National Research, Development and Innovation Office.}\and G\'abor Dam\'asdi\thanks{Supported by the \'{U}NKP-21-3 New National Excellence Program of the Ministry for Innovation and Technology from the source of the National Research, Development and Innovation fund, by the ERC Advanced Grant GeoScape, by the Lendület program of the Hungarian Academy of Sciences (MTA), under the grant LP2017-19/2017 and by the Thematic Excellence Program TKP2021-NKTA-62 of the National Research, Development and Innovation Office}\and Bal\'azs Keszegh\thanks{Supported by the National Research, Development and Innovation Office -- NKFIH under the grant K 132696 and FK 132060, by the \'UNKP-21-5, \'UNKP-22-5 and \'UNKP-23-5 New National Excellence Program of the Ministry for Innovation and Technology from the source of the National Research, Development and Innovation Fund, by the J\'anos Bolyai Research Scholarship of the Hungarian Academy of Sciences, by the ERC Advanced Grant ``ERMiD'', by the EXCELLENCE-24 project no.~151504 of the NRDI Fund, by the Lendület program of the Hungarian Academy of Sciences (MTA), under the grant LP2017-19/2017 and by the Thematic Excellence Program TKP2021-NKTA-62 of the National Research, Development and Innovation Office.}\and D\"om\"ot\"or P\'alv\"olgyi\thanks{Supported by the ÚNKP-22-5 and ÚNKP-23-5 New National Excellence Program of the Ministry for Innovation and Technology from the source of the National Research, Development and Innovation Fund, by the Lendület program of the Hungarian Academy of Sciences (MTA), under the grant LP2017-19/2017, by the Thematic Excellence Program TKP2021-NKTA-62 of the National Research, Development and Innovation Office and by the János Bolyai Research Scholarship of the Hungarian Academy of Sciences, by the ERC Advanced
Grant “ERMiD” and by the EXCELLENCE-24 project no. 151504 of the NRDI Fund.}}
\affiliation{ELTE E\"{o}tv\"{o}s Lor\'{a}nd University, Budapest, Hungary\\
HUN-REN Alfr\'{e}d R\'{e}nyi Institute of Mathematics, Budapest, Hungary}


\begin{document}
\publicationdata{vol. 27:3}{2025}{2}{10.46298/dmtcs.15019}{2025-01-03; 2025-01-03; 2025-06-24}{2025-06-30}

\maketitle

\begin{abstract}	
	We study systems of orientations on triples (that is, an assignment \begin{math}\o\end{math} of a value from \begin{math}\{+1,-1,0\}\end{math} to each ordered triple) that satisfy the following so-called interiority condition: \begin{math}\operatorname{\circlearrowleft}(ABD)=\operatorname{\circlearrowleft}(BCD)=\operatorname{\circlearrowleft}(CAD)=1\end{math} implies \begin{math}\operatorname{\circlearrowleft}(ABC)=1\end{math} for any \begin{math}A, B, C, D\end{math}.
    We call such an orientation a partial 3-order, a natural generalization of a poset that has several interesting special cases.
    As an example, the well-known order type of a planar point set (that can have collinear triples) is a partial 3-order. 
    
    In our previous paper ``Orientation of convex sets'' we defined a partial 3-order on pairwise intersecting convex sets that we call 3-orders realizable by convex sets. A good cover is a family of compact closed sets in the plane such that the intersection of the members of any subfamily is either contractible or empty. In this paper, we extend the partial 3-order from the previous one and define a partial 3-order on good covers  having pairwise intersecting sets.
    
   If the family is non-degenerate with respect to the orientation, i.e., always $\o(ABC)\ne 0$, we obtain a total 3-order.  The main result of this paper is that there is a total 3-order, which is realizable by points that is not realizable by good covers, implying also that it is not realizable by convex sets. This latter problem was left open in our earlier paper.
    Our proof involves several combinatorial and geometric observations that can be of independent interest.
    Along the way, we regard the 3-orders realizable by various special good covers, in particular, by families of topological trees, curves, lines and Y-shapes that pairwise intersect exactly once.
\end{abstract}  

\section{Introduction} 
Given some base set, a mapping $\o$ from its ordered triples to \begin{math}\{+1,-1,0\}\end{math} is a \emph{partial orientation} if
\begin{displaymath}\o(A,B,C)=\o(C,A,B)=\o(B,C,A)=-\o(A,C,B)=-\o(B,A,C)=-\o(C,B,A)\text{ for every } A, B, C.\end{displaymath}
 In the special case, when $\o$ nowhere vanishes, i.e., it is never $0$, we call such an assignment a \emph{total orientation}. We will also use the term \emph{orientation of the ordered $3$-tuple \begin{math}(A, B, C)\end{math}} for the value \begin{math}\o (A, B, C)\end{math}, which we will oftentimes only denote by $\o(ABC)$.
An orientation satisfies the \emph{interiority condition} if \begin{displaymath}\o(ABD)=\o(BCD)=\o(CAD)=1\text{ implies }\o(ABC)=1\text{ for every }A, B, C, D.\end{displaymath}
Let $conv(ABC)$ denote the set of those $D$-s for which \begin{math}\o(ABD)=\o(BCD)=\o(CAD)=1\end{math} or \begin{math}\o(ABD)=\o(BCD)=\o(CAD)=-1\end{math}. See Figure \ref{fig:convexdef}(a).
Note that in the definition of \begin{math}conv(ABC)\end{math} the order of \begin{math}A,B,C\end{math} is not relevant and that for the standard orientation of points in general position it coincides with the standard notion of \begin{math}D\in conv(ABC)\end{math}. However, in general, it does not necessarily satisfy all natural properties of convexity.

\begin{figure}[t]
	\centering
	\includegraphics[width=12cm]{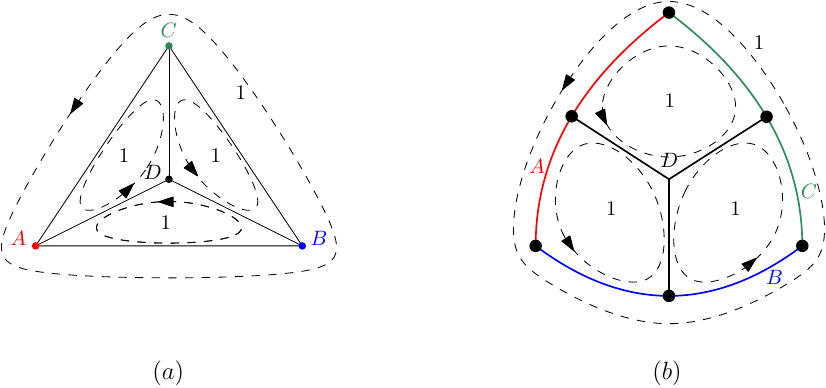}
	\caption{The interiority condition is true for points (a) and also for pairwise intersecting sets that form a good cover (b); here \begin{math}\o(ABD)=\o(BCD)=\o(CAD)=\o(ABC)\end{math}, or equivalently, we can write \begin{math}D\in conv(ABC)\end{math}.}
	\label{fig:convexdef}
\end{figure}

A partial orientation that satisfies the interiority condition is called in the literature an \emph{order of rank 3}, we will refer to it as a \emph{partial 3-order} or \PO. A total orientation that satisfies the interiority condition is called a \emph{total 3-order} or \TO. Orders of higher rank can also be defined, we refer to \cite{C-3TO} for a more detailed introduction to this topic. In the remainder when we talk about some order it is always a 3-order, although some of the following definitions and statements hold for orders of other ranks.

The notion of a total 3-order was introduced by \cite{Knuth} under the name \emph{interior triple system}, according to Knuth ``for want of a better name.'' He noted that taking the orientations (in the well-known geometric sense) of all triples of a planar point set in general position we get a total 3-order, while if we allow collinearity, we get a partial 3-order. More recently, \cite{BFSchSchS} have also studied total 3-orders. Unaware of Knuth's work, they named them \emph{generalized signotopes}, and studied primarily a special subclass of them that can be defined from topological drawings of $K_n$. They studied the appropriately defined versions of well-known results and proved, for example, that Kirchberger's theorem holds for total 3-orders, but Helly's theorem does not. They also rediscovered Knuth's construction that there are \begin{math}2^{\Omega(n^3)}\end{math} different \TO's, and using recursion they obtained a slightly better constant and also an upper bound, proving that the number of \TO's is between \begin{math}2^{0.25\binom n3+o(n^3)}\end{math} and \begin{math}2^{0.84\binom n3+o(n^3)}\end{math}.

If there is a bijection between the members of two families of objects such that each ordered $3$-tuple has the same orientation, then we say that they have the same order type. In general, once an order has been fixed (or is implicitly understood) on a (typically finite) family $\mathcal F$, then this order is called the \emph{order type} of $\mathcal F$ so the order type of a family of objects for us is just the underlying partial order. We also say that the order is \emph{realized} by $\mathcal F$. Traditionally, this term is used for point sets (see \cite{GP93}); we use it in a more general sense, for any family (with a fixed order).
Note that in the literature order type is often only used for points in general position; we also allow the sign to be $0$, as in the more restrictive case we can emphasize that the order type is a total order.

We define certain classes of partial and total orders that are order types of specific families (in other words, realizable by specific families).
First, we say that a \TO (resp.\ \PO) that has a realization by a planar set of points is a \pTO (resp.\ \pPO).
We denote the family of all \PO's by \cPO and, similarly, for its subfamilies, we use the calligraphic \cTO, \cpPO, \cpTO, respectively.

In a companion paper \cite{C-3TO}, motivated by a lemma of \cite{JKLPT} (see also \cite{LT}), we have defined an orientation on pairwise intersecting planar closed convex sets, as follows.
If \begin{math}A\cap B\cap C\ne \emptyset\end{math}, then $\o(ABC)=0$.
Otherwise, by \cite{JKLPT}, \begin{math}\R^2\setminus (A\cup B\cup C)\end{math} has one bounded component, and its boundary has exactly one arc from each of the boundaries of $A$, $B$ and $C$.
We defined \begin{math}\o(ABC)=1\end{math} if in cyclic counterclockwise order these arcs belong to \begin{math}A,B,C\end{math}, and defined \begin{math}\o(ABC)=-1\end{math} in the remaining cases. We showed that $\o$ satisfies the interiority condition, i.e., it is a 3-order.
Denote the subfamily of \cTO and \cPO that have a realization by pairwise intersecting planar convex sets by \cCTO and \cCPO, respectively. We refer to them as total/partial 3-orders realizable by convex sets (without explicitly adding that the convex sets are pairwise intersecting).
In particular, if no three sets from a pairwise intersecting convex family have a common point (called a \emph{holey family} in \cite{C-3TO}), the orientation $\o$ gives a \CTO on them.

In this paper, we extend the orientation $\o$ to good covers having pairwise intersecting sets, and denote the respective families by \cGCTO and \cGCPO. We refer to them as total/partial 3-orders realizable by good covers (without explicitly adding that the sets are pairwise intersecting). We call a family of non-empty closed sets in the Euclidean plane a \emph{good cover} if the intersection of any subfamily is contractible or empty, see \cite{Weil}\footnote{We note that in the literature the underlying space can vary and also closedness is not assumed or openness is assumed instead. Therefore, from now on whenever we refer to a planar set it is always assumed that it is a non-empty closed subset of the Euclidean plane even if not said explicitly.}. For example, any family of (closed) convex sets is a good cover. Another example of good covers is any family of curves that pairwise intersect in at most one point, which can either be a crossing point or a tangency.
Note that each set in a good cover needs to be connected as it is contractible.

As convex sets are always good covers, we have $\cCPO\subset \cGCPO\subset \cPO$.
Both containments are strict:
$\cGCPO\ne \cPO$ follows from Theorem \ref{thm:main}, while
$\cCPO\ne\cGCPO$ holds as the order type of a certain five-point configuration can be realized with good covers whose \pPO was shown not be a \CPO in \cite{C-3TO}; see Figure \ref{fig:squaremid}.

\begin{figure}[t]
	\centering
	\includegraphics[width=12cm]{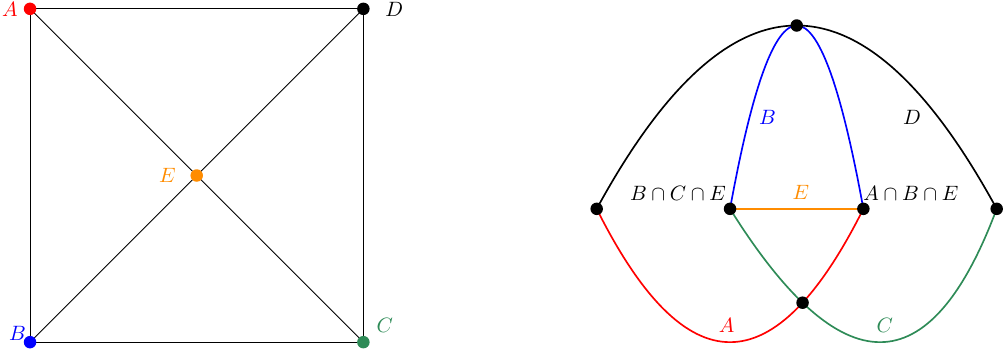}
	\caption{An order type realizable by points, realized on the left by the four vertices of a square (\begin{math}A,B,C,D\end{math}) and their center ($E$). It was shown in \cite{C-3TO} that this order type is not realizable by convex sets. However the right figure is a realization by a good cover.}
	\label{fig:squaremid}
\end{figure}

We have shown in \cite{C-3TO} that there exits a total 3-order realizable by  convex sets which is not realizable by points and vice versa, i.e., \begin{math}\cCTO \nsubseteq \cpTO\end{math} and \begin{math}\cpPO \nsubseteq \cCPO\end{math}.
The first statement \begin{math}\cCTO \nsubseteq \cpTO\end{math} implies also \begin{math}\cCPO \nsubseteq \cpPO\end{math} and thus \begin{math}\cGCPO \nsubseteq \cpPO\end{math}, but this latter statement also easily follows from our \ref{item:5points} of Proposition \ref{prop:tr} and the example given in \ref{fig:toptreeS3}(c).

In this paper we establish the strengthening that there is a 3-order that is realizable by points in general position which is not realizable by convex sets, i.e., $\cpTO \nsubseteq \cCTO$.
This follows from the following more general result.

\begin{thm}\label{thm:main}
    $\cpTO \nsubseteq \cGCTO$, i.e., there exists a total 3-order realizable by points which is not realizable by good covers.
\end{thm}

Our proof will first establish that the 3-order of some point set is not realizable by some special subfamily of (pairwise intersecting) good covers, and then gradually increase the complexity of this subfamily, while also making our point set larger, until we establish the theorem. One important subfamily we will introduce is the family of topological trees for which we prove that the total 3-orders realizable by topological trees that pairwise intersect exactly once are the same as the ones realizable by pairwise intersecting good covers. One key difference between $\pTO$  and $\GCTO$ that we can exploit is that in the former case we can easily add another copy of an element, while in the latter case we cannot; see also \cite{C-3TO}, Theorem~6 and Problem~14.

\smallskip
The rest of the paper is organized as follows. In Section \ref{sec:gc} we define an orientation $\o$ on good covers, and show that in a \GCTO realization we can assume that the sets are pairwise once intersecting topological trees, which also helps to prove that $\o$ satisfies the interiority condition, thus a 3-order. In Section \ref{sec:subfamilies} we define the subfamilies of \cTO that are of interest to us and depict their relations to each other in Figure \ref{fig:3TOposet}. In Section \ref{sec:T} we determine which \pTO's are realizable by T-shapes. In Section \ref{sec:tangencies} we establish an auxiliary lemma on the tangency graph of topological trees. In Section \ref{sec:main} we present the proof of Theorem \ref{thm:main}. Finally, in Section \ref{sec:discussion} we pose some open problems.

\begin{figure}[t]
	\centering
	\includegraphics[width=12cm]{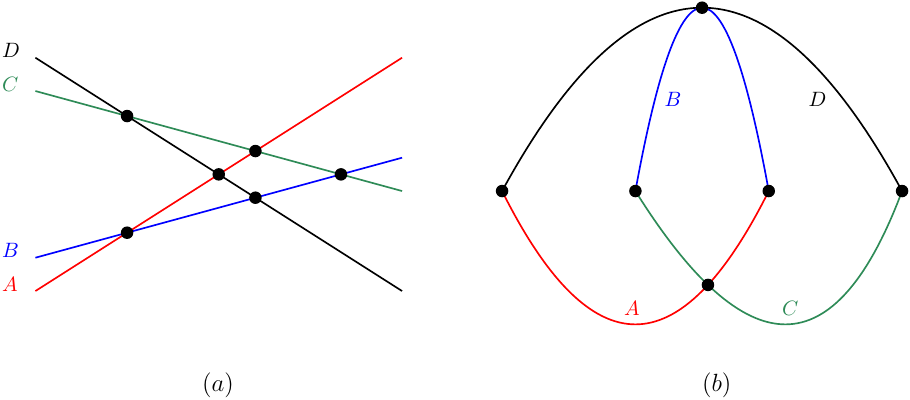}
	\caption{Two good covers that consist of pairwise intersecting topological trees.
		The four sets depicted in each figure are such that no point is contained in three sets, so their orientation is a \GCTO, and no three sets satisfy the premise of the interiority condition:	\begin{math}\o(BCD)=\o(CDA)=\o(DAB)\end{math}.
		Note that during the closed walk \begin{math}A\cp B\end{math}--\begin{math}B\cp C\end{math}--\begin{math}C\cp D\end{math}--\begin{math}D\cp A\end{math}--\begin{math}A\cp B\end{math} in (a) we wind around once, while in (b) we wind around twice.}
	\label{fig:doublecircle}
\end{figure}

\section{Good covers and topological trees}\label{sec:gc}

Given three pairwise intersecting sets from a good cover, $A$, $B$, $C$, define \begin{math}\o(A,B,C)\end{math} as follows.
If \begin{math}A\cap B\cap C\ne\emptyset\end{math}, then \begin{math}\o(A,B,C)=0\end{math}.
Otherwise, a version of the Nerve Theorem that is about good covers \cite[Theorem~13.4]{BT} states that \begin{math}A\cup B\cup C\end{math} is homotopy equivalent to a triangle and then by the Jordan Curve Theorem it follows that \begin{math}A\cup B\cup C\end{math} is a closed set whose complement is the disjoint union of two connected open sets, one bounded and one unbounded. Let $o$ be an arbitrary point of the bounded set. Assume that there are three points \begin{math}z\in A\cap B\end{math}, \begin{math}x\in B\cap C\end{math} and \begin{math}y\in C\cap A\end{math} and directed paths \begin{math}\gamma_A\subseteq A\end{math} connecting $y$ with $z$, \begin{math}\gamma_B\subseteq B\end{math} connecting $z$ with $x$ and  \begin{math}\gamma_C\subseteq C\end{math} connecting $x$ with $y$ that together form the directed closed curve $\gamma$. Let \begin{math}\o(A,B,C)\end{math} be the winding number\footnote{Informally, the winding number is the number of times that the curve goes counterclockwise around the point, it is negative if the curve goes around the point clockwise.} of $\gamma$ around $o$ \footnote{Note that for convex sets this definition coincides with the one from \cite{C-3TO}.}.
We will show that the above orientation $\o$ is well-defined (see Claim \ref{claim:welldef}) and satisfies the interiority condition (see Corollary \ref{cor:gcispo}), thus it is a partial 3-order.
For an example that satisfies the premise of the interiority condition, see Figure \ref{fig:convexdef}(b), while no three of the four sets in Figures \ref{fig:doublecircle}(a) and (b) satisfy the premise.

First we fix some notations. When we consider a realization of a 3-order, we make no distinction between an element and the set representing it.
For \begin{math}n>3\end{math}, the restriction of the orientation $\o$ to some elements \begin{math}\cX=\{X_1,\dots,X_n\}\end{math} from a family (point sets, good covers, etc.), is denoted by \begin{math}\o(X_1,\dots,X_n)\end{math} or simply \begin{math}\o(\cX)\end{math}.

\begin{claim}\label{claim:welldef}
    The orientation $\o$ is well-defined. 
\end{claim}

This simple topological claim is probably already known, but we outline a proof for the sake of completeness.

        \begin{figure}[t]
	\centering
	\includegraphics[width=8cm]{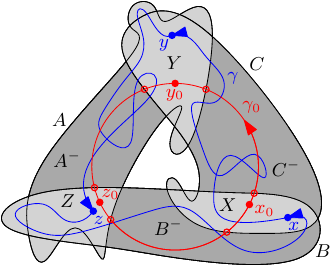}
	\caption{The sets \begin{math}(A^-,X,B^-,Y,C^-,Z)\end{math} and the canonical curve $\gamma_0$ in the proof of Claim \ref{claim:welldef}.}
	\label{fig:newtopproof}
\end{figure}

\begin{proof}
    Suppose that we are given three pairwise intersecting sets \begin{math}A,B,C\end{math} with \begin{math}A\cap B\cap C=\emptyset\end{math}. Define $A^-$ to be the unique connected component of \begin{math}A\setminus (B\cup C)\end{math} which has a common boundary with both \begin{math}A\cap C\end{math} and \begin{math}A\cap B\end{math} (such a component exists). Define $B^-$ and $C^-$ similarly. Now let $X$ be the union of $B\cap C$ and those connected components of $A\setminus C$ that do not contain $A^-$ and those connected components of $C\setminus A$ that do not contain $C^-$. It is easy to see that, as it is a good cover, $X$ must be simply connected. We define $Y$ and $Z$ similarly. See Figure \ref{fig:newtopproof}. 
    We now concentrate on the above defined $6$ simply connected sets in the following cyclic order: \begin{math}(A^-,Z,B^-,X,C^-,Y)\end{math}. They are pairwise disjoint, while the boundaries of any two consecutive sets intersect in a connected set. Further, while the boundaries of \begin{math}A^-,B^-,C^-\end{math} might share points, the boundaries of \begin{math}X,Y,Z\end{math} are pairwise disjoint.\footnote{We note that from this structure it also follows that the complement of  \begin{math}A\cap B\cap C=\emptyset\end{math} is the union of a bounded and an unbounded component, which we showed earlier using the Nerve Theorem.} Thus there exists a canonical closed directed curve $\gamma_0$ that is formed by $6$ directed paths, each lying inside \begin{math}A^-,Z,B^-,X,C^-,Y\end{math}, in this order (the paths corresponding to \begin{math}X,Y,Z\end{math} might be only single points). We can easily split $\gamma_0$ into $3$ paths, $\gamma_{0,A},\gamma_{0,B},\gamma_{0,C}$ each lying in \begin{math}A,B,C\end{math}, respectively, in this order, with splitting points \begin{math}x_0\in X, y_0\in Y, z_0\in Z\end{math}.
               
    Recall that in the definition of the orientation $\gamma$ is as follows: \begin{math}z\in A\cap B\end{math}, \begin{math}x\in B\cap C\end{math} and \begin{math}y\in C\cap A\end{math} and directed paths \begin{math}\gamma_A\subseteq A\end{math} connecting $y$ with $z$, \begin{math}\gamma_B\subseteq B\end{math} connecting $z$ with $x$ and  \begin{math}\gamma_C\subseteq C\end{math} connecting $x$ with $y$ together form the directed closed curve $\gamma$. As $A$ is simply connected, every path inside it is homotopy equivalent, therefore there is a homotopy between $\gamma_A$ and $\gamma_{0,A}$ inside $A$. Similarly we get a homotopy between $\gamma_B$ and $\gamma_{0,B}$ inside $B$ and between $\gamma_C$ and $\gamma_{0,C}$ inside $C$. These together define a homotopy between $\gamma$ and $\gamma_0$ inside \begin{math}A\cup B\cup C\end{math}, which thus preserves the winding number around $o$. Thus indeed the winding number of $\gamma$ is always the same irrespective of our choice of points and curves. Finally, the winding number of $\gamma_0$ is either $+1$ or $-1$. Indeed, $\gamma_0$ is a Jordan Curve (no self-intersections) so it cannot have winding number more than $2$ or less than $-2$. Also, it has a homotopy to \begin{math}A\cup B\cup C\end{math} inside \begin{math}A\cup B\cup C\end{math} (which surrounds $o$), so its winding number cannot be $0$.
\end{proof}

We define a special subfamily of good covers, where each set is a topological tree.

A \emph{topological tree} is an injective embedding of a (graph theoretic) tree, that has no degree two vertices, into the plane, such that vertices are mapped to points, and edges are mapped to simple curves. The images of the degree one vertices are called \emph{leaves}, while the images of the vertices with degree at least three are called \emph{branching points}.

A family of topological trees forms a good cover if every pair of trees intersects at most once.
(This is not an if and only if condition, but it will be more convenient for us to work only with such families.)
For two trees $A$ and $B$ we denote their intersection point by \begin{math}A\cp B\end{math}, i.e., \begin{math}A\cap B=\{A\cp B\}\end{math}.\footnote{The difference is that there is less space around the new math operator. The similarity can lead to no confusion, as these two operators denote practically the same thing.}
For three trees, $A$, $B$ and $C$, we have \begin{math}\o(A,B,C)= 0\end{math} if and only if their pairwise intersection points coincide.

\begin{figure}[t]
	\centering
	\includegraphics[width=6cm]{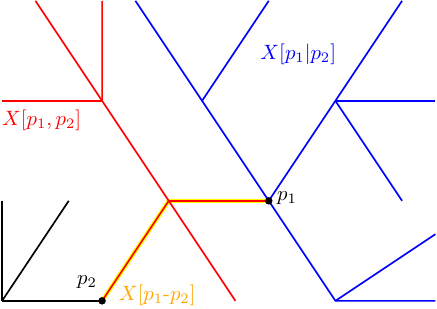}
	\caption{Parts of the topological tree $X$.}
	\label{fig:xdef}
\end{figure}

We need to introduce some notation. See Figure \ref{fig:xdef} for illustration of the next definition.

\begin{defi}
Suppose that $X$ is a topological tree and the point $p_i$ is in $X$ for each \begin{math}1\le i\le k\end{math}.

Define \begin{math}X[p_1\mbox{-}\dots\mbox{-}p_k]\end{math} to be the minimal connected subset of $X$ which contains $p_i$ for every $i$.
In particular, \begin{math}X[p_1\mbox{-}p_2]\end{math} is the path connecting $p_1$ and $p_2$ in $X$.
Note that \begin{math}X[p_1\mbox{-}\dots\mbox{-} p_k]=\cup_{i,j} X[p_i\mbox{-} p_j]\end{math}.

Define \begin{math}X[p_1,\dots, p_k]\end{math} to be the minimal connected subset $X'$ of $X$ which contains $p_i$ for every $i$, and for which every connected component of $X\setminus X'$ has $p_i$ on its boundary for some $i$.

If \begin{math}A_1,\dots A_k\end{math} are topological trees that intersect $X$ once and \begin{math}p_i=X\cp A_i\end{math}, then for brevity we can replace $p_i$ in the above notations with $A_i$.
For example, \begin{math}X[A_1\mbox{-}A_2]=X[p_1\mbox{-}A_2]=X[p_1\mbox{-}p_2]\end{math}.
\end{defi}

Now we are ready to prove our main structural tool. 

\begin{prop}\label{prop:tr}
    Assume that in a planar family:
\end{prop}  

\begin{seclist}
\item\label{item:toptree} Every set is a topological tree.
    
\smallskip
\item\label{item:1pointintersection} Every pair of sets intersects in exactly one point.\footnote{Two trees are allowed to have multiple branches from their intersection point.}
    
\smallskip
\indent\textit{Then the following hold:}

\smallskip
\item\label{item:jordancurve} The union of any three sets, \begin{math}A,B,C\end{math}, without a common point, contains exactly one cycle, i.e., a Jordan curve. The complement of the cycle has two connected parts, one bounded and one unbounded, by the Jordan Curve Theorem.
The bounded one we call the \emph{hollow} and is denoted by
\begin{math}\hollow(ABC)\end{math}.\footnote{Similarly to hollows of convex sets as defined in \cite{C-3TO}.}
The boundary of \begin{math}\hollow(ABC)\end{math} consists of parts of \begin{math}A,A\cp B,B,B\cp C,C,C\cp A\end{math}, in this order, if \begin{math}\o(ABC)=1\end{math}.
From the boundary of \begin{math}\hollow(ABC)\end{math}, there can be subtrees of \begin{math}A,B,C\end{math} going inwards and outwards, these we call \emph{hairs}.
Hairs might have branchings on them, but they are disjoint from each other. See Figure \ref{fig:toptreeS3}(a).

\smallskip
\item\label{item:intcondition} Any four sets satisfy the interiority condition, thus $\o$ is a partial 3-order.

\item\label{item:hairpartition} Further, if \begin{math}D\in conv(ABC)\end{math}, then
\begin{displaymath}
    \hollow(ABD) \disjunion \hollow(BCD) \disjunion \hollow(ACD)\disjunion D[A\mbox{-}B\mbox{-}C]\cup\{\textit{the union of at most one hairpart from each of }A,B\textit{ and }C\}
\end{displaymath}
gives a partition of \begin{math}\hollow(ABC)\cup\{A\cp D, B\cp D, C\cp D\}\end{math}.

In particular, we have \begin{math}D[A,B,C]\setminus \{A\cp D, B\cp D, C\cp D\}\subset\hollow(ABC)\end{math}.
See Figure \ref{fig:toptreeS3}.

\smallskip
\item\label{item:5points} If \begin{math}D,E\in conv(ABC)\end{math} and the orientation on \begin{math}A,B,C,D,E\end{math} is realizable by five points in general position, then \begin{math}D\cp E\in \hollow(ABC).\end{math}
\end{seclist}

\begin{figure}[t]
	\centering
	\includegraphics[width=14cm]{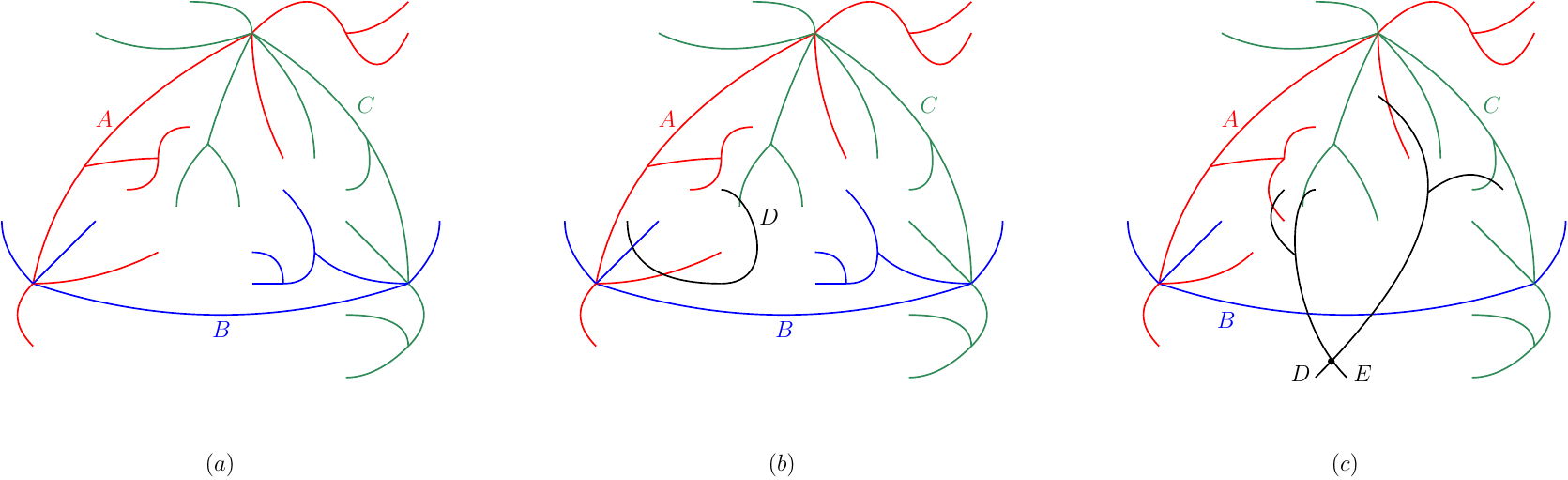}
	\caption{\begin{math}\protect\hollow(ABC)\end{math} and the hairs.}
	\label{fig:toptreeS3}
\end{figure}

By \ref{item:jordancurve} and \ref{item:intcondition}, the orientation $\o$ defined for good covers gives a partial 3-order for any family satisfying the conditions \ref{item:toptree} and \ref{item:1pointintersection}.
We call a partial 3-order that is realizable this way a \TrPO.
If in the realization in addition no three trees have a common intersection, then it is a \TrTO.

\begin{rem}
    Note that because of the hairs several intuitive statements are false. For example, it is possible that \begin{math}D[A,B,C]\subset\hollow(ABC)\end{math} but \begin{math}D\notin conv(ABC)\end{math}. See Figure \ref{fig:toptreeS3}(b).
    Also, if in \ref{item:5points} we only assume that \begin{math}D,E\in conv(ABC)\end{math}, then it is possible that \begin{math}D\cp E\notin \hollow(ABC)\end{math}. See Figure \ref{fig:toptreeS3}(c).
\end{rem}

\begin{proof}[\ref{item:jordancurve}.]
    As $A$ is a tree, there is exactly one path in $A$ between \begin{math}A\cp B\end{math} and \begin{math}A\cp C\end{math}, and this cannot intersect $B$ or $C$.
    Similarly, there is one path in $B$ between \begin{math}A\cp B\end{math} and \begin{math}B\cp C\end{math}, and there is one path in $C$ between \begin{math}A\cp C\end{math} and \begin{math}B\cp C\end{math}.
    The union of these three paths gives the required Jordan curve.    
\end{proof}

\begin{proof}[of \ref{item:intcondition} and \ref{item:hairpartition}.]

    Assume that \begin{math}D\in conv(ABC)\end{math}, then without loss of generality \begin{math}\o(ABD)=\o(BCD)=\o(CAD)=1\end{math}. To show that the interiority condition holds we need to prove that \begin{math}\o(ABC)=1\end{math}.
    First, we assume \begin{math}\o(ABC)\ne 0\end{math}, i.e., \begin{math}A\cp B\cp C=\emptyset\end{math}.
    
    Consider \begin{math}D'=D[A\mbox{-}B\mbox{-}C]\end{math}.
    Either $D'$ is a path, or a Y-shaped star, i.e. the image of a tree with a single 3-degree vertex and no vertices of degree more than 3. 
    
    Assume first that $D'$ is a path, without loss of generality \begin{math}B\cp D\in D[A\mbox{-} C]=D'\end{math}, i.e., \begin{math}B\cp D\end{math} lies between \begin{math}A\cp D\end{math} and \begin{math}C\cp D\end{math} on $D$.
    From \begin{math}\o(CAD)=1\end{math}, we know on which side of $D'$ the hollow \begin{math}\hollow(ACD)\end{math} lies.
    See Figure \ref{fig:toptreeS4}(a).
    
    Note that \begin{math}\hollow(ABD)\subset \hollow(ACD)\end{math} would imply \begin{math}\o(ABD)=\o(ACD)\ne\o(CAD)\end{math}, contradicting our assumptions.
    Similarly, \begin{math}\hollow(BCD)\not\subset \hollow(ACD)\end{math}.
    
    This implies that none of the paths \begin{math}B[D\mbox{-}A]\end{math} and \begin{math}B[D\mbox{-}C]\end{math} can start from \begin{math}B\cp D\end{math} towards the interior of \begin{math}\hollow(ACD)\end{math}, as otherwise they would need to intersect \begin{math}\partial\hollow(ACD)\end{math}, either in $A$ or in $C$, which would give \begin{math}\hollow(ABD)\subset \hollow(ACD)\end{math} or \begin{math}\hollow(BCD)\subset \hollow(ACD)\end{math}, respectively, contradicting our previous observation. See Figure \ref{fig:toptreeS4}(a).
    As both paths start from \begin{math}B\cp D\end{math} towards the exterior of \begin{math}\hollow(ACD)\end{math}, we can regard $D'$ as a (degenerate) Y-shape such that its leaves in the counterclockwise order are \begin{math}A\cp D,B\cp D,C\cp D\end{math}.
    
    The same argument rules out the possibility that $D'$ is a Y-shape such that its leaves in the counterclockwise order are \begin{math}A\cp D,C\cp D,B\cp D\end{math}.
    
    Therefore, we can conclude that $D'$ needs to be a (possibly degenerate) Y-shape such that its leaves in the counterclockwise order are \begin{math}A\cp D,B\cp D,C\cp D\end{math}.
    Denote the branching point of $D'$ by $D_y$ (where \begin{math}D_y=B\cp D\end{math} if $D'$ is degenerate). See Figure \ref{fig:toptreeS4}(b).

\begin{figure}[t]
	\centering
	\includegraphics[width=14cm]{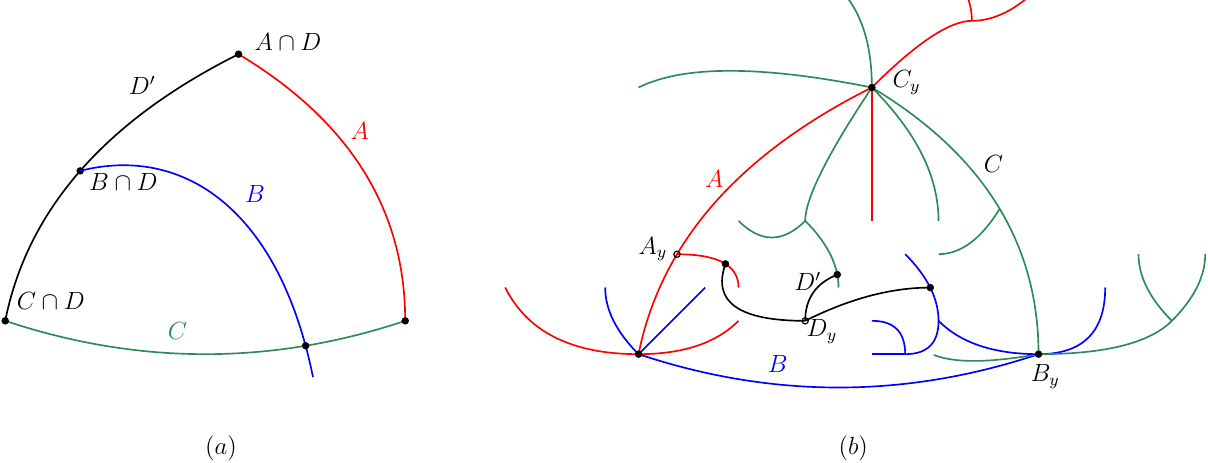}
	\caption{Proof of \ref{item:intcondition} and \ref{item:hairpartition}.}
	\label{fig:toptreeS4}
\end{figure}

    Denote the branching point of \begin{math}A[B\mbox{-}C\mbox{-}D]\end{math} by $A_y$, or if \begin{math}A[B\mbox{-}C\mbox{-}D]\end{math} is a path, then let $A_y$ stand for whichever of \begin{math}A\cp B,A\cp C\end{math} and \begin{math}A\cp D\end{math} lies in the middle of the path.
    In other words, $A_y$ is the point up to which \begin{math}A[D\mbox{-}B]\end{math} and \begin{math}A[D\mbox{-}C]\end{math} follow the same route starting from \begin{math}A\cp D\end{math}.    
    In particular, \begin{math}A[D\mbox{-}A_y]\end{math} does not contain \begin{math}A\cp B\end{math} or \begin{math}A\cp C\end{math} in its interior.
    We similarly define $B_y$ and $C_y$.
    
    Now, walk along \begin{math}\partial\hollow(ABD),\partial \hollow(BCD)\end{math} and \begin{math}\partial\hollow(ACD)\end{math}, starting always from $D_y$.
    Note that in the union of these thee walks, during the walk around \begin{math}\partial\hollow(ABD)\end{math}, we cover the part from $D_y$ to $A_y$, then we go from $A_y$ to $B_y$, then back to $D_y$.
    Around \begin{math}\partial\hollow(BCD)\end{math}, we cover the part from $D_y$ to $B_y$, then we go from $B_y$ to $C_y$, then back to $D_y$.
    Finally, around \begin{math}\partial\hollow(ACD)\end{math}, we cover the part from $D_y$ to $C_y$, then we go from $C_y$ to $A_y$, then back to $D_y$.
    Note that each part from $D'$ occurs there and back, while all the other parts are disjoint, apart from their endpoints.
    But this means that by eliminating the there-and-back parts, we get a walk from $A_y$ to $B_y$ to $C_y$, then back to $A_y$, that has the same orientation as the original walks, and contains \begin{math}A\cp B,B\cp C\end{math} and \begin{math}A\cp C\end{math}. That is, \begin{math}\o(ABC)=1\end{math}, as claimed. This finishes the proof of the interiority condition if \begin{math}\o(ABC)\ne 0\end{math}.
    
    Now assume for a contradiction that \begin{math}\o(ABC)=0\end{math}, i.e., \begin{math}A\cp B\cp C\ne \emptyset\end{math}.
    By \ref{item:1pointintersection}, \begin{math}A\cp B\cp C\end{math} is a single point; denote it by $x$.
    If we replace $A$ with \begin{math}A[x\mbox{-}D]\end{math}, then \begin{math}\o(A,B,C,D)\end{math} remains the same.
    We can similarly replace $B$ and $C$ with \begin{math}B[x\mbox{-}D]\end{math} and \begin{math}C[x\mbox{-}D]\end{math}, respectively.
    This way each of $A$, $B$ and $C$ became a curve ending in $x$.
    But then by a slight perturbation of the curves in the vicinity of $x$ we can achieve that they pairwise intersect once but in three different points, such that \begin{math}\o(ABC)\end{math} becomes $-1$ (see Figure \ref{fig:tripleintersection}). But this would contradict the interiority condition in the case \begin{math}\o(ABC)\ne 0\end{math}, which we have already proved.
    This finishes the proof of the interiority condition if \begin{math}\o(ABC)= 0\end{math}.
        
    \begin{figure}[t]
    	\centering
    	\includegraphics[width=14cm]{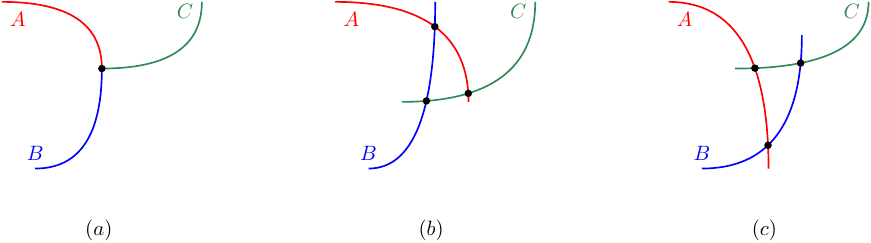}
    	\caption{A triple intersection can be perturbed in two ways.}
    	\label{fig:tripleintersection}
    \end{figure}
        
    The structural description that we have obtained in the \begin{math}\o(ABC)\ne 0\end{math} case implies that
    \begin{align*}
        \hollow(ABC) & = \hollow(ABD) \disjunion \hollow(BCD) \disjunion \hollow(ACD) \disjunion D'\cup \\
        & \cup\{\text{the part of }A\text{ from }A_y\text{ to }A\cp D,\text{ the part of }B\text{ from }B_y\text{ to }B\cp D,\text{ the part of }C\text{ from }C_y\text{ to }C\cp D\}
    \end{align*}
    --- this last part gives the three possible hairsparts.
    Since $D$ intersects \begin{math}\hollow(ABC)\end{math} exactly three times, this also implies \begin{math}D[A,B,C]\setminus \{A\cp D, B\cp D, C\cp D\}\subset\hollow(ABC)\end{math}.
\end{proof}

\begin{proof}[of \ref{item:5points}.]   
    Suppose for a contradiction that \begin{math}x=D\cp E\notin \hollow(ABC)\end{math}.
    
    By \ref{item:hairpartition}, \begin{math}D[A,B,C]\subset \hollow(ABC)\end{math} and thus $x$ falls in one connected component of \begin{math}D\setminus D[A,B,C]\end{math} which implies that the three paths \begin{math}D[E\mbox{-}A]\end{math}, \begin{math}D[E\mbox{-}B]\end{math}, \begin{math}D[E\mbox{-}C]\end{math} all go the same way from $x$ until they reach \begin{math}\hollow(ABC)\end{math}. Denote their intersection point with \begin{math}\partial \hollow(ABC)\end{math} by $D_x$. Note that $D_x$ is one of \begin{math}A\cp D,B\cp D,C\cp D\end{math}.
    
    Similarly, the three paths \begin{math}E[D\mbox{-}A]\end{math}, \begin{math}E[D\mbox{-}B]\end{math}, \begin{math}E[D\mbox{-}C]\end{math}, all start the same way from $x$.
    Denote their intersection point with \begin{math}\partial \hollow(ABC)\end{math} by $E_x$.

    We claim that \begin{math}\o(ADE)=\o(BDE)=\o(CDE)\end{math}.
    Indeed, this follows from the fact that \begin{math}\hollow(ADE)\end{math}, \begin{math}\hollow(BDE)\end{math} and \begin{math}\hollow(CDE)\end{math} are all contained in the union of \begin{math}\hollow(ABC)\end{math} and the topological triangle whose vertices are \begin{math}x,D_x,E_x\end{math}.
    To see this, note that for any of these hollows, $x$ will be a vertex, while the two sides of the hollow adjacent to $x$ will go through $D_x$ and $E_x$, respectively, and then continue inside \begin{math}\hollow(ABC)\end{math} until they reach the other two vertices of the hollow, because of \ref{item:hairpartition}.
    
    But this contradicts that \begin{math}D,E\in conv(ABC)\end{math} and the orientation on \begin{math}A,B,C,D,E\end{math} is realizable by five points in general position, as if in this realization the points $D$ and $E$ are contained in the convex hull of $A$, $B$ and $C$, then two of these three points will fall on different sides of the $DE$ line, so \begin{math}\o(ADE)=\o(BDE)=\o(CDE)\end{math} is not possible.
\end{proof}

Next, we prove that $\o$ behaves essentially the same way on any good cover as on topological trees.

\begin{lem}\label{lem:gc2tree}
    The sets in any good cover, where at most one triple has a non-empty intersection, can be replaced by topological trees that pairwise intersect at most once, such that $\o$ remains unchanged on all triples. 
\end{lem}

\begin{proof}
    See Figure \ref{fig:goodcover2tree}(a) for an illustration of the proof.
    Note that we did not assume that the family is pairwise intersecting, thus the orientation may be undefined for some triples, in which case it will remain undefined.
    
    We can assume that every set intersects some other set from our family.
    For each set $A$ from the good cover, and for each connected component $A_i$ of those points that are only in $A$ and in no other set, we do the following.
    On the boundary of $A_i$ there are pairwise disjoint arcs which are on the boundary of some other set as well. 
    We put a topological star inside $A_i$ whose center is on one of these arcs, and there is a leaf on each of the rest of the arcs. 
    
    Now, assume that there is no triple intersection.
    For each non-empty intersection \begin{math}A\cap B\end{math} it is simply connected as the family is a good cover. We select a point $p_{AB}$ inside it ($p_{AB}$ will be the intersection point of the topological trees corresponding to $A$ and $B$), and draw non-crossing curves from $p_{AB}$, one to every leaf of the earlier defined stars that are on the boundary of \begin{math}A\cap B\end{math}. Each such curve is added to the topological tree it touches at the boundary of \begin{math}A\cap B\end{math}.
    
    If there is a triple intersection \begin{math}A\cap B\cap C\ne\emptyset\end{math}, then
    we treat \begin{math}(A\cap B) \cup (A\cap C)\cup (B\cap C)\end{math} as one double intersection, and do the same as before, selecting a point $p_{ABC}$ from \begin{math}A\cap B\cap C\end{math} (and no other points from \begin{math}A\cap B,A\cap C\end{math}, and \begin{math}B\cap C\end{math}, so in this case there are no points $p_{AB}$, $p_{AC}$, $p_{BC}$).
    In the remainder of the proof, we will not discuss this special triple intersection region in detail---all steps work for it the same way.

    \begin{figure}[t]
    	\centering
    	\includegraphics[width=12cm]{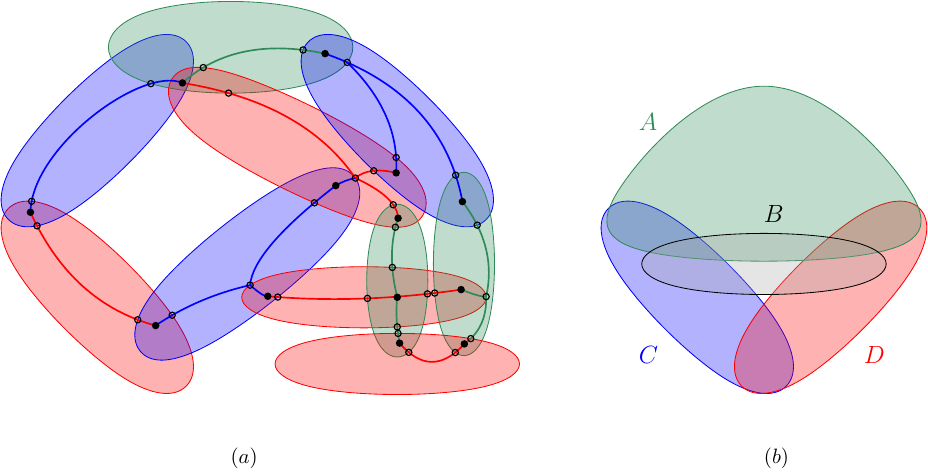}
    	\caption{(a) A good cover redrawn with topological trees that pairwise intersect in at most one point. (b) A good cover that cannot be redrawn this way.}
    	\label{fig:goodcover2tree}
    \end{figure}
    
    It is easy to see that we get topological trees, as all the points inside a set $A$ are eventually connected and no cycle can be created.
    Denote this tree by \begin{math}T_A\subset A\end{math}.
    If \begin{math}A\cap B\ne \emptyset\end{math} then the corresponding topological trees intersect  exactly once, in the point selected inside the intersection of their corresponding sets, i.e., \begin{math}T_A\cp T_B=p_{AB}\end{math}. Otherwise the corresponding topological trees are disjoint.
    
    We are left to show that the orientations are preserved. If three sets were not pairwise intersecting then the corresponding trees are also not such. Otherwise, the $3$ corresponding trees satisfy \ref{item:1pointintersection} and $\o$ is well-defined on them by \ref{item:jordancurve}. If they have a common intersection, then the corresponding three trees do too (and so $\o$ is preserved when it is equal to $0$). Finally, let $A,B,C$ be three pairwise intersecting sets with no triple-intersection from the good cover. Take the (unique) Jordan curve $\gamma$ that is the union of three curves \begin{math}\gamma_A\subset T_A\end{math}, \begin{math}\gamma_B\subset T_B\end{math} and \begin{math}\gamma_C\subset T_C\end{math}.
    As \begin{math}T_X\subset X\end{math} for any set $X$, also \begin{math}\gamma_A\subset A\end{math}, \begin{math}\gamma_B\subset B\end{math} and \begin{math}\gamma_C\subset C\end{math}, so the same $\gamma$ shows that \begin{math}\o(T_AT_BT_C)=\o(ABC)\end{math}.    
\end{proof}   

\begin{rem}
    The condition that at most one triple has a non-empty intersection is necessary because we allow trees to intersect in at most one point.
    If, for example, \begin{math}A\cap B\cap C\ne\emptyset\end{math} and \begin{math}A\cap B\cap D\ne\emptyset\end{math} but \begin{math}A\cap B\cap C\cap D=\emptyset\end{math}, then this obviously cannot be realized by trees that pairwise intersect at most once. See Figure \ref{fig:goodcover2tree}(b) for such a good cover.
    But this is essentially the only obstruction---our proof can be modified in a straightforward way to work also if we require that there are no four sets such that \begin{math}A\cap B\cap C\ne\emptyset\end{math} and \begin{math}A\cap B\cap D\ne\emptyset\end{math}.
\end{rem}

\begin{cor}\label{cor:gcispo}
    The orientation $\o$ is a partial 3-order on good covers having pairwise intersecting sets.
\end{cor}
\begin{proof}
    We need to show that $\o$ satisfies the interiority condition.
    Take four sets such that \begin{math}\o(ABD)=\o(BCD)=\o(CAD)=1\end{math}.
    In particular, these four sets can have at most one non-empty triple intersection, \begin{math}A\cap B\cap C\end{math}.
    Therefore, by Lemma \ref{lem:gc2tree} we can convert \begin{math}A, B, C, D\end{math} to trees while preserving $\o$, and so \ref{item:intcondition} implies \begin{math}\o(ABC)=1\end{math}.
\end{proof}

\begin{rem}
    This also implies that $\o$ is a partial 3-order on convex sets, reproving a key lemma from \cite{C-3TO}.
\end{rem}

Denote a partial/total 3-order realizable by good covers having pairwise intersecting sets (resp.\ by pairwise $1$-intersecting topological trees), by \GCPO/\GCTO (resp.\ by \TrPO/\TrTO), and the respective subfamilies by calligraphic, as usual. As mentioned earlier, we refer to such 3-orders as being realizable by good covers without explicitly adding that it has pairwise intersecting sets, as the orientation is not defined otherwise anyway. Similarly, we refer to 3-orders realizable by topological trees without explicitly adding that they are pairwise $1$-intersecting. We follow this convention also for other subfamilies of good covers defined later on.

\begin{cor}\label{cor:GCTr}
    $\cGCTO= \cTrTO$, i.e. the total 3-orders realizable by good covers are the same as the total 3-orders realizable by topological trees.
\end{cor}

This implies that to establish Theorem \ref{thm:main}, it is enough to prove $\cpTO \subsetneq \cTrTO$, that is, to find a 3-order realizable by points in general position which is not realizable by topological trees.

\section{Subfamilies of \cTO}\label{sec:subfamilies}

\def\kicsiny{0.8}
\begin{figure}[p]
        	\begin{center}
			\begin{minipage}{.24\textwidth}
				\centering
				\scalebox{\kicsiny}{\input{Figures/p1}}
			\end{minipage}
			\begin{minipage}{.24\textwidth}
				\centering
				\scalebox{\kicsiny}{\input{Figures/p2}}
			\end{minipage}
			\begin{minipage}{.24\textwidth}
				\centering
				\scalebox{\kicsiny}{\input{Figures/p3}}
			\end{minipage}
			\begin{minipage}{.24\textwidth}
				\centering
				\scalebox{\kicsiny}{\input{Figures/p4}}
			\end{minipage}
		\end{center}
		\vspace{-20pt}
		\begin{center}
			\begin{minipage}{.24\textwidth}
				\centering
				\scalebox{\kicsiny}{\input{Figures/g1}}
			\end{minipage}
			\begin{minipage}{.24\textwidth}
				\centering
				\scalebox{\kicsiny}{\input{Figures/g2}}
			\end{minipage}
			\begin{minipage}{.24\textwidth}
				\centering
				\scalebox{\kicsiny}{\input{Figures/g3}}
			\end{minipage}
			\begin{minipage}{.24\textwidth}
				\centering
			    \scalebox{\kicsiny}{\input{Figures/g4}}
			\end{minipage}
		\end{center}
	\vspace{-5pt}
		\hrule
	\vspace{-10pt}

		\begin{center}
			\begin{minipage}{.24\textwidth}
				\centering
				\scalebox{\kicsiny}{\input{Figures/p5}}
			\end{minipage}
			\begin{minipage}{.24\textwidth}
				\centering
				\scalebox{\kicsiny}{\input{Figures/p6}}
			\end{minipage}
			\begin{minipage}{.24\textwidth}
				\centering
				\scalebox{\kicsiny}{\input{Figures/p7}}
			\end{minipage}
			\begin{minipage}{.24\textwidth}
				\centering
				\scalebox{\kicsiny}{\input{Figures/p8}}
			\end{minipage}
		\end{center}
		\vspace{-20pt}
		\begin{center}
			\begin{minipage}{.24\textwidth}
				\centering
				\scalebox{\kicsiny}{\input{Figures/g5}}
			\end{minipage}
			\begin{minipage}{.24\textwidth}
				\centering
				\scalebox{\kicsiny}{\input{Figures/g6}}
			\end{minipage}
			\begin{minipage}{.24\textwidth}
				\centering
				\scalebox{\kicsiny}{\input{Figures/g7}}
			\end{minipage}
			\begin{minipage}{.24\textwidth}
				\centering
				\scalebox{\kicsiny}{\input{Figures/g8}}
			\end{minipage}
		\end{center}
	\vspace{-5pt}
		\hrule
	\vspace{-10pt}
		\begin{center}
			\begin{minipage}{.24\textwidth}
				\centering
				\scalebox{\kicsiny}{\input{Figures/p9}}
			\end{minipage}
			\begin{minipage}{.24\textwidth}
				\centering
				\scalebox{\kicsiny}{\input{Figures/p10}}
			\end{minipage}
			\begin{minipage}{.24\textwidth}
				\centering
				\scalebox{\kicsiny}{\input{Figures/p11}}
			\end{minipage}
			\begin{minipage}{.24\textwidth}
				\centering
				\scalebox{\kicsiny}{\input{Figures/p12}}
			\end{minipage}
		\end{center}
		\vspace{-20pt}
		\begin{center}
			\begin{minipage}{.24\textwidth}
				\centering
				\scalebox{\kicsiny}{\input{Figures/g9}}
			\end{minipage}
			\begin{minipage}{.24\textwidth}
				\centering
				\scalebox{\kicsiny}{\input{Figures/g10}}
			\end{minipage}
			\begin{minipage}{.24\textwidth}
				\centering
				\scalebox{\kicsiny}{\input{Figures/g11}}
			\end{minipage}
			\begin{minipage}{.24\textwidth}
				\centering
				Not a \gTO?
			\end{minipage}
		\end{center}
	\vspace{-5pt}
		\hrule
	\vspace{-10pt}
		\begin{center}
			\begin{minipage}{.24\textwidth}
				\centering
				\scalebox{\kicsiny}{\input{Figures/p13}}
			\end{minipage}
			\begin{minipage}{.24\textwidth}
				\centering
				\scalebox{\kicsiny}{\input{Figures/p14}}
			\end{minipage}
			\begin{minipage}{.24\textwidth}
				\centering
				\scalebox{\kicsiny}{\input{Figures/p15}}
			\end{minipage}
			\begin{minipage}{.24\textwidth}
				\centering
				\scalebox{\kicsiny}{\input{Figures/p16}}
			\end{minipage}
		\end{center}
		\vspace{-25pt}
		\begin{center}
			\begin{minipage}{.24\textwidth}
				\centering
				\scalebox{\kicsiny}{\input{Figures/g13}}
			\end{minipage}
			\begin{minipage}{.24\textwidth}
				\centering
				\scalebox{\kicsiny}{\input{Figures/g14}}
			\end{minipage}
			\begin{minipage}{.24\textwidth}
				\centering
				\scalebox{\kicsiny}{\input{Figures/g15}}
			\end{minipage}
			\begin{minipage}{.24\textwidth}
				\centering
				\scalebox{\kicsiny}{\input{Figures/g16}}
			\end{minipage}
		\end{center}
		\vspace*{-3pt}
		\caption{Out of the $16$ order types of points on $6$ points, we could realize $15$ with curves.}
		\label{fig:g3TO}
\end{figure}

Here we define several subfamilies of good covers and point sets that are of interest to us.
For completeness, we also include the already defined families.
The first two families were already defined by \cite{Knuth}, while the third is the usual order type on planar points.
The containment relations among the respective \cTO families are depicted in Figure \ref{fig:3TOposet}.

\begin{itemize}
    \item{\TO:} A total 3-order over any set, i.e., an orientation on the ordered triples that satisfies the interiority condition.

    \item{CC system:} A total 3-order over any set that satisfies the interiority condition and the following transitivity condition. If  \begin{math}\o(ABC)=\o(ABD)=\o(ABE)=\o(AED)=\o(ADC)\end{math}, then they are equal to \begin{math}\o(AEC)\end{math}.
    CC systems are in correspondence to abstract order types which encode pseudoconfigurations of points. 
    
    \item{\pTO:} A total 3-order on a planar point set in general position, where each ordered triple is oriented depending on whether the points are in clockwise or counterclockwise position.
    This is more commonly known as the order type of the point set.
    
    \item{\fastTO:} Special case of \pTO where the points form a fast-growing point set. (Defined in Section \ref{sec:T}.)
    
    \item{\convexpTO:} Special case of \pTO where the points are in convex position. (All \convexpTO's of some fixed size are isomorphic up to relabeling.)
    
    \item{\GCTO:} A 3-order realizable by a good cover of pairwise intersecting sets in the plane, such that no three of them have a point in common, where each triple is oriented according to $\o$ defined in Section \ref{sec:gc}.
    
    \item{\CTO:} Special case of \GCTO where each set is convex. (A family of such convex sets is called a holey family and studied in detail in \cite{C-3TO}.)
    
    \item{\TrTO:} Special case of \GCTO where each set is a topological tree.
    
    \item{\YTO:} Special case of \TrTO where each topological tree is a \emph{Y-shape}, meaning that it has only one branching point, which is of degree three, i.e., it is an image of $K_{1,3}$.
    We also require for every pair of topological trees that they either cross (as opposed to touch) or one of them ends in their intersection point.\footnote{Such intersections could also be elongated to form a crossing. We do allow the branching point of a Y-shape to also be its intersection point with another Y-shape; this is relevant when the intersection point is the branching point for both trees. For every Y-shape this can occur with at most one other Y-shape, as no three have a point in common.} 
    
    \item{\TTO:} Special case\footnote{Strictly speaking, T-shapes are not topological trees because they are unbounded, but it would be easy to turn them into Y-shapes by replacing the lines/halflines with sufficiently long segments.} of \YTO where each set is a \emph{T-shape}, i.e., one horizontal line, and one vertical downward halfline, ending in an interior point of the horizontal line.
    Note that no lines or halflines coincide in the representation because any two T-shapes intersect in exactly one point.
    
    \item{\gTO:} Special case\footnote{Strictly speaking, curves are not Y-shapes because they do not have a branching point, but we could introduce one on each at an arbitrary place to turn them into Y-shapes.} of \YTO where each set is a simple curve, and these curves pairwise cross (as opposed to touch).
    We have depicted the representation of some \pTO's as \gTO's in Figure \ref{fig:g3TO}.
    
    \item{\lTO:} Special case of \gTO where each curve is a line, i.e., a collection of lines in general position (no two are parallel and no three pass through the same point)

     \item{Topological drawings of the complete graph:} Orientations given by simple topological planar drawings of a complete graph, studied by \cite{BFSchSchS}. 
\end{itemize}

\begin{figure}
    \centering
    \includegraphics[width=14cm]{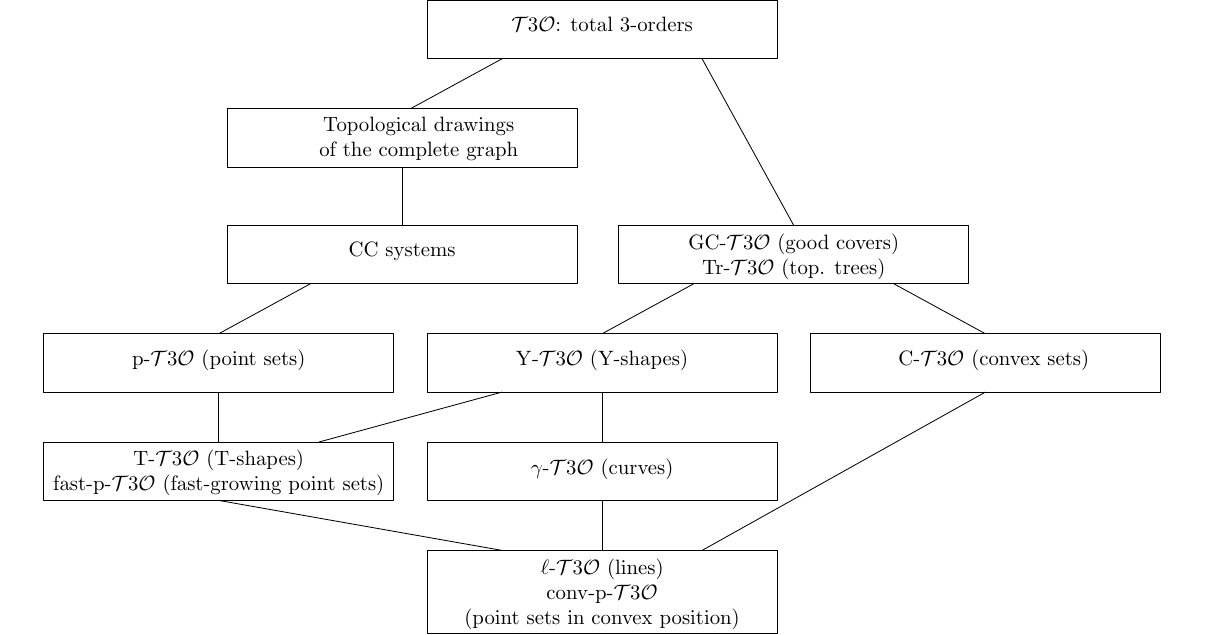}
    \caption{Containment relations among different \cTO's.
    Most containments are true by definition, our main result, Theorem \ref{thm:main}, is the non-containment between \cpTO and \cGCTO.
    In most cases we did not prove that the containments are strict, but we conjecture that all of them are.
    The equivalence $\cGCTO= \cTrTO$ is Corollary \ref{cor:GCTr}, $\cTTO= \cfastTO$ is Proposition \ref{prop:fastT}, and $\clTO= \cconvexpTO$ is Corollary \ref{cor:lines}.}
    \label{fig:3TOposet}
\end{figure}

\begin{obs}[\cite{GP84}]
	If we have lines \begin{math}l_1,\ldots,l_n\end{math} ordered according to their slopes in clockwise circular order and points \begin{math}p_1,\ldots,p_n\end{math} in convex position, ordered in counterclockwise order, then \begin{math}\o(l_1,\ldots,l_n)=\o(p_1,\ldots,p_n)\end{math}.
\end{obs}

\begin{cor}\label{cor:lines}
    $\clTO= \cconvexpTO$.
\end{cor}

\section{T-shapes and fast-growing point sets}\label{sec:T}
Call a total 3-order on $n$ elements \emph{fast-growing} if its elements have an ordering \begin{math}\{p_1,\dots,p_n\}\end{math} such that there is a permutation \begin{math}\pi\colon \{1,\dots,n\}\to\{1,\dots,n\}\end{math} such that for every $i<j<k$ we have that: \begin{math}\o(p_ip_jp_k)=-1\end{math} if and only if \begin{math}\pi(j)>\pi(i),\pi(k)\end{math}. We denote the family of fast-growing total 3-orders by \cfastTO. Note that for some fixed $\pi$ this is the total 3-order defined by the point set \begin{math}\{(i,N_{\pi(i)}) : 1\le i\le n\}\end{math} if $N_i$ is a sufficiently fast-growing function, like $2^i$. 
We call this a canonical representation of this fast-growing total 3-order. Note that \begin{math}p_{\pi^{-1}(1)},\dots, p_{\pi^{-1}(n)}\end{math} is the order of these points according to their $y$-coordinates. We also call a point set fast-growing if its order type is a fast-growing total 3-order. We note that this definition of fast-growing point sets is similar to and a restriction of the definition of decomposable point sets by \cite{balko}.
An anonymous referee also pointed out that fast-growing point sets are a subclass of the dual of shellable arrangements, introduced by \cite{AFORSV}, which appeared after the first version of our paper.

We have seen for each $\pi$ such a fast-growing point set. In particular, if $\pi$ is the identity permutation, we get that all point sets in convex position are fast-growing, as their order type is the same as for the point set \begin{math}\{(i,2^i) : 1\le i\le n\}\end{math}.
Note that \begin{math}p_1,p_n,p_{\pi(1)}\end{math} and $p_{\pi(n)}$ are always (not necessarily different) extremal points of \cP.
It is easy to check that any point set on at most five points is fast-growing, while on six points out of the 16 \pTO's 12 are fast-growing; see Figure \ref{fig:fast3TO} and its caption.

\begin{figure}
		\begin{center}
			\begin{minipage}{.24\textwidth}
				\centering
				\input{Figures/p1sz}
			\end{minipage}
			\begin{minipage}{.24\textwidth}
				\centering
				\input{Figures/p2sz}
			\end{minipage}
			\begin{minipage}{.24\textwidth}
				\centering
				\input{Figures/p3sz}
			\end{minipage}
			\begin{minipage}{.24\textwidth}
				\centering
				\input{Figures/p4}
			\end{minipage}
		\end{center}
		
		\begin{center}
			\begin{minipage}{.24\textwidth}
				\centering
			\input{Figures/p5sz}
			\end{minipage}
			\begin{minipage}{.24\textwidth}
				\centering
			\input{Figures/p6}
			\end{minipage}
			\begin{minipage}{.24\textwidth}
				\centering
			\input{Figures/p7sz}
			\end{minipage}
			\begin{minipage}{.24\textwidth}
				\centering
		        \input{Figures/p8sz}
			\end{minipage}
		\end{center}
		
		\begin{center}
			\begin{minipage}{.24\textwidth}
				\centering
				\input{Figures/p9sz}
			\end{minipage}
			\begin{minipage}{.24\textwidth}
				\centering
				\input{Figures/p10sz}
			\end{minipage}
			\begin{minipage}{.24\textwidth}
				\centering
				\input{Figures/p11sz}
			\end{minipage}
			\begin{minipage}{.24\textwidth}
				\centering
				\input{Figures/p12sz}
			\end{minipage}
		\end{center}
		
		\begin{center}
			\begin{minipage}{.24\textwidth}
				\centering
				\input{Figures/p13}
			\end{minipage}
			\begin{minipage}{.24\textwidth}
				\centering
				\input{Figures/p14}
			\end{minipage}
			\begin{minipage}{.24\textwidth}
				\centering
				\input{Figures/p15sz}
			\end{minipage}
			\begin{minipage}{.24\textwidth}
				\centering
				\input{Figures/p16sz}
			\end{minipage}
		\end{center}
		\caption{Out of the $16$ order types of points on $6$ points, exactly $12$ are fast growing.
		For a representation with points of these $12$ order types we have numbered the points such that the numbers indicate the order of the points from bottom to top in a canonical representation, i.e. the point $p_i$ gets the label $\pi(i)$. The order of the points themselves required to check if it defines a fast-growing total 3-order, i.e., the $x$-coordinate order in a canonical representation, is not explicitly given but can be easily determined by creating a canonical representation, placing the points one-by-one, each time high enough, in the order given by the labels.
		A simple case analysis shows that in the remaining $4$ cases, whose points are labeled with letters, their order types are not fast growing, and by Proposition \ref{prop:fastT}, also not \TTO's.} 
		\label{fig:fast3TO}
	\end{figure}

Corollary \ref{cor:lines} implies that $\clTO\subset\cfastTO$.
This containment is strict because, for example, four points in non-convex position is a fast-growing point-set yet its order type cannot be realized by lines.
The growth rates are also different: we have $(n-1)!$ \lTO's, about $(n!)^2$ \fastTO's, and about $(n!)^4$ \pTO's on $n$ labeled elements, see \cite{GP93,Handbook}.

\begin{figure}[t]
	\centering
	\includegraphics[width=14cm]{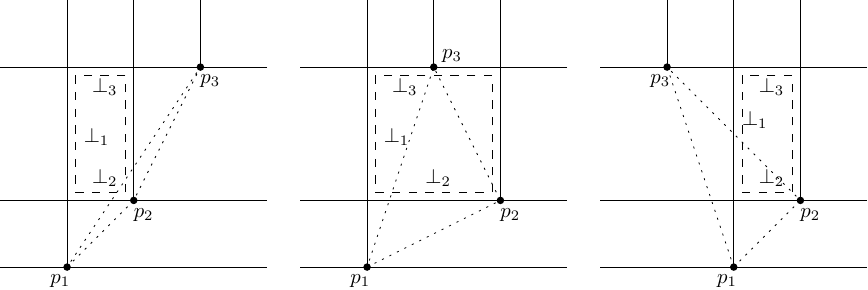}
	\caption{The correspondence between the 3-order of $\perp$-shapes and fast-growing point sets. The cases when the orientation of the triple is $+1$ are listed here, the rest are symmetrical.}
	\label{fig:fastgrowing}
\end{figure}

We can also give the following characterization.

\begin{prop}\label{prop:fastT}
    $\cfastTO= \cTTO$.
\end{prop}
\begin{proof}
    First, observe that the 3-orders defined by T-shapes are the same as defined by $\perp$-shapes (shapes that we get from T-shapes by reflecting on the $x$-axis) and so it is enough to show that the latter are the same as \fastTO's.
    
    Given a \fastTO, we can realize it with a fast-growing point set, and then for each point put a $\perp$-shape whose hat and leg meet at that point. 
    It is easy to see that the orientation of this family of $\perp$-shapes is the same as for the point set, see Figure \ref{fig:fastgrowing}. 
    
    Conversely, given a \TTO, we realize it with $\perp$-shapes. Notice that changing the $y$-coordinates of the hats, while preserving their order, does not change the 3-order of the $\perp$-shapes. Thus, we can change the $y$-coordinates so that they form a fast-growing function. Take the meeting point of the hat and leg for each $\perp$-shape; these form a fast-growing point set. As before, it is easy to see that the order type of this point set is the same as the orientation of the original family of $\perp$-shapes.
\end{proof}

\begin{cor}\label{cor:Tnotinp}
    $\clTO\subsetneq \cTTO\subsetneq \cpTO$.
\end{cor}

\section{The number of tangencies in a family of \texorpdfstring{$t$}{t}-intersecting trees}\label{sec:tangencies}
Before we can start the proof of Theorem \ref{thm:main}, we need one more tool, in order to be able to switch from tangencies to crossings among the sets of some large enough family at the cost of losing some sets.
The following lemma was recently shown by \cite{KP22} in the special case when we have curves instead of trees; our proof uses several ideas from their proof.

\begin{lem}\label{lem:treetangencies}
	In the tangency graph of a family of topological trees that pairwise intersect in at most $t$ points, and no three trees have a common point, there is no $K_{t+3,c}$ for some large enough $c$ (depending only on $t$).
\end{lem}

\begin{proof}
	Suppose for a contradiction that there is a family of topological trees as in the lemma such that its tangency graph contains $K_{t+3,c}$ as a subgraph. Call the $t+3$ trees that form one part of this bipartite graph \emph{red}, and the $c$ trees that form the other part of the bipartite graph \emph{blue}. Each red tree is cut by the other $t+2$ red trees into at most $1+(t+2)t$ parts, and by the pigeonhole principle there are at least \begin{math}c_2=c/(1+(t+2)t)^{t+3}\end{math} blue trees that intersect the same part from each red tree. Choosing these red parts, from now on we assume that we have a family of red and blue trees whose tangency graph contains a $K_{t+3,c_2}$ and the red trees are pairwise disjoint. After small local redrawing operations of the blue trees we can also assume that each tangency point is a leaf of the respective blue tree. Furthermore, from each blue tree we keep only the edges that participate in some path connecting two tangency points (of this tree and one of the $t+3$ red trees). 

	Thus, the blue trees each have exactly $t+3$ leaves, corresponding to tangency points, which implies that they have at most $t+1$ branching points. 
	As no three trees have a common point, each of these branching points can belong to at most one other blue tree.
    Define a graph on the $c_2$ blue trees where two trees are connected by an edge if they intersect in a branching point of one of them.
    Because of our previous observation, this graph can have at most $c_2(t+1)$ edges.
    Therefore, we can select an independent set of size \begin{math}c_3=c_2/(2t+1)\end{math} from it by Tur\'an's theorem.
		
	There are only bounded many options (depending on $t$) how such a blue tree can be embedded in the plane. Formally, if we label the leaves with the red trees touching at that leaf and put two blue trees in the same equivalence class whenever they are topologically equivalent (including the labels) and also have the same rotation system (i.e., the edges leaving each branching point come in the same circular order), then we get a bounded number of equivalence classes.
	Thus, if $c_3$ is big enough, then there are two blue trees in the same equivalent class---denote them by $B$ and $B'$.

\begin{figure}[t]
	\centering
	\includegraphics[width=8cm]{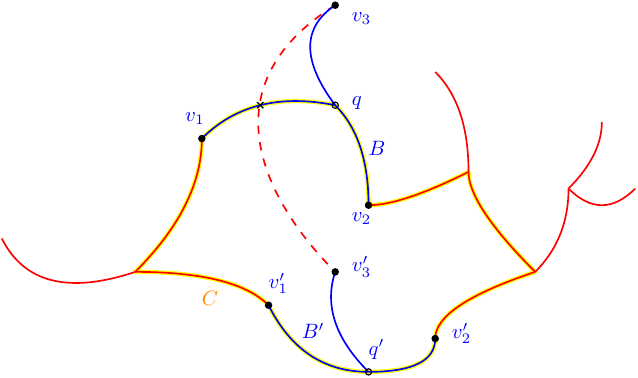}
	\caption{If there is no intersection of the two blue trees on one of the paths connecting $v_1$ with $v_2$ and $v_1'$ with $v_2'$, then we get a contradiction, as at $q$ and $q'$ one goes inside the yellow Jordan curve $C$, the other outside $C$.}
	\label{fig:bluetrees}
\end{figure}
	
	To summarize, we are given two trees, $B$ and $B'$, both with $t+3$ labeled leaves, such that $B$ and $B'$ are topologically equivalent (including their labels) and their rotation systems are also the same.
	We are left to show that such a pair of trees intersects in at least $t+1$ points, as this gives the desired contradiction. We show this by induction on the number of leaves of $B$ and $B'$. For $t+1=0$ this trivially holds.
	Next, we assume that $t\ge 0$ and that for $t-1$ the statement holds, and show that it also holds for $t$. 
	
	First, we claim that there are two leaves, $v_1$ and $v_2$, of $B$ such that after removing from $B$ the path $P$ connecting $v_1$ and $v_2$ (without removing the branching points that lie on $P$), we get a connected graph (a tree). To see that, take any branching point $s$ of $B$ (which exists as we have $t+3\ge 3$ leaves) and take a branching point $q$ which is farthest from $s$ in $B$ in the combinatorial distance (i.e., the connecting path has the largest possible number of branching points); by definition, $q$ is equal to $s$ if there are no other branching points. The at least two edges incident to $q$ that do \emph{not} go towards $s$ must lead to leaves, as $q$ was a farthest branching point. It is easy to see that any two of these leaves are good choices for $v_1$ and $v_2$. Take also the corresponding leaves $v_1',v_2'$ and branching point $q'$ in $B'$ and the path $P'$ connecting them. See Figure \ref{fig:bluetrees}.
	
	Next, we show that $P$ must intersect $B'$ or $P'$ must intersect $B$. Assume on the contrary. Observe that the endpoints of $P$ and $P'$ are tangency points with the same pair of red trees. $P$, $P'$ and the two red trees together partition the plane into two connected regions, such that $q$ and $q'$ are on their common boundary, a Jordan curve, which we denote by $C$. Take a third leaf, $v_3$, of $B$ and a path $Q$ connecting $q$ and $v_3$. Similarly, let $Q'$ be the path of $B'$ connecting $q'$ with the leaf $v_3'$ corresponding to $v_3$. Due to the indirect assumption, both $Q$ and $Q'$ are disjoint from $C$. As the rotation system around $q$ and $q'$ is the same, they leave $C$ towards a different side of $C$, and therefore must end (in $v_3$ and $v_3'$) on a different side of $C$. However, $v_3$ and $v_3'$ are connected by a red tree that is disjoint from $C$, a contradiction.
	   
	Thus, without loss of generality, we can assume that $P$ and $B'$ intersect in a point $p$. Notice that $p$ cannot be equal to $q$, as it is a branching point of $B'$. If $p$ is on the path of $B$ connecting $v_1$ with $q$, then we delete this path from $B$ and the path connecting $v_1'$ with $q'$ in $B'$. Otherwise, we delete the path of $B$ connecting $v_2$ with $q$ and the path of $B'$ connecting $v_2'$ with $q$.
	In both cases, the two trees we get each have $t-1$ leaves, they do not intersect anymore in $p$, and they are also topologically equivalent (with the leaves labeled, and the rotation system preserved), thus we can apply induction to find $t$ intersection points between them. Together with $p$, there are $t+1$ intersection points between $B$ and $B'$, a contradiction.
\end{proof}

\section{Proof of Theorem \ref{thm:main}}\label{sec:main}

Now we are ready to prove Theorem \ref{thm:main}.
Our starting point is the following geometric observation.

\begin{figure}[t]
	\centering
	\includegraphics[width=14cm]{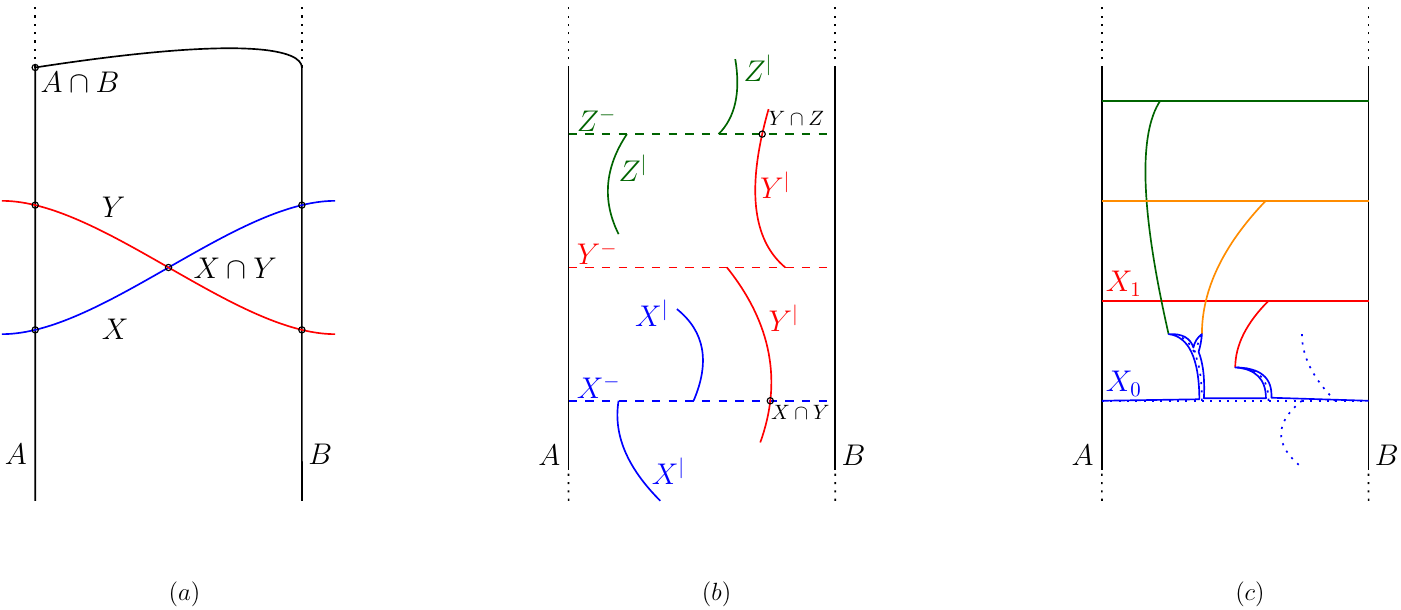}
	\caption{(a) Proposition \ref{prop:halfline}, (b),(c) Lemma \ref{lem:halfline}.}
	\label{fig:halfline}
\end{figure}

\begin{prop}\label{prop:halfline}
    Let $A$ and $B$ be two disjoint (topological) halflines. Let $X$ and $Y$ be two topological trees that have one intersection point. Assume further that both $X$ and $Y$ intersect both $A$ and $B$ in exactly one point, these four intersection points are all different, and none of them are branching points of the trees. 
    Assume also that \begin{math}X[A\mbox{-}B]\end{math} and \begin{math}Y[A\mbox{-}B]\end{math} 
    leave each of the two halflines on the same side, i.e., they touch the halflines from the same direction.
    
    If \begin{math}\o(XYA)\ne\o(XYB)\end{math}, then \begin{math}X\cp Y=X[A,B] \cp Y[A,B]\end{math} and \begin{math}X\cp Y\end{math} cannot be a crossing point between \begin{math}X[A\mbox{-}B]\end{math} and \begin{math}Y[A\mbox{-}B]\end{math}.\footnote{However, it can be a touching point between them.}
\end{prop}
\begin{proof}
    We can extend $A$ into a larger topological tree by adding a curve to it that connects $A$ and $B$ but does not intersect $X$ or $Y$. It not hard to see that we also have enough freedom to choose this curve such that in the resulting family we have \begin{math}\o(XAB)=\o(XYA)\end{math}. Since \begin{math}\o(XYA)\ne\o(XYB)\end{math}, without loss of generality, \begin{math}\o(XAB)=\o(XYA)=\o(XBY)\end{math}, i.e., \begin{math}X\in conv(ABY)\end{math}.
    From \ref{item:hairpartition}, \begin{math}X[A,B,Y]\setminus \{A\cp X, B\cp X, Y\cp X\}\subset \hollow(ABY)\end{math}, which implies \begin{math}X\cp Y=X[A,B] \cp Y[A,B]\end{math}.
    
    To see the second part, assume on the contrary that \begin{math}X\cp Y=X[A\mbox{-}B]\cp Y[A\mbox{-}B]\end{math} is a crossing point of $X$ and $Y$. Then \begin{math}\o(XYA)=\o(XYB)\end{math}, a contradiction. See Figure \ref{fig:halfline}(a).
\end{proof}

\begin{lem}\label{lem:halfline}
    Suppose that we are given two halflines, $A$ and $B$, and a realization of a $\TrTO$ with a family of topological trees \cP. If every pair \begin{math}X,Y\in \cP\end{math} satisfies the conditions of Proposition \ref{prop:halfline} with $A$ and $B$, then the order type of some \begin{math}\cQ\subset \cP\end{math}, \begin{math}|\cQ|\ge |\cP|/2\end{math} has a realization with $T$-shapes, i.e., is a \TTO.
\end{lem}
\begin{proof}
    We can assume that $A$ and $B$ are vertical halflines, and that \begin{math}X=X[A,B]\end{math} for all \begin{math}X\in\cP\end{math}, as other parts of $X$ do not participate in any intersection. By Proposition \ref{prop:halfline} \begin{math}X[A\mbox{-}B]\end{math} and \begin{math}Y[A\mbox{-}B]\end{math} are pairwise non-crossing. After a local redrawing around the touching points, we can also assume that there are no touchings. Thus, we can assume that each \begin{math}X[A\mbox{-}B]\end{math} is horizontal and each $X$ is contained inside the strip bounded by the lines of these halflines.
    We can also assume that $A$ is downward infinite. $B$ is either also downward infinite, or upward infinite. As each \begin{math}X\in \cP\end{math} is inside the strip, we can cover both cases by simply assuming that both $A$ and $B$ are lines bounding the strip.
    We can order the elements of $\cP$ by the heights of \begin{math}X[A\mbox{-}B]\end{math}, i.e., define $X<Y$ if the $y$-coordinate of \begin{math}A\cp X\end{math} (and \begin{math}B\cp X\end{math}) is less than the $y$-coordinate of \begin{math}A\cp Y\end{math} (and \begin{math}B\cp Y\end{math}).
    
    For simplicity, denote \begin{math}X[A\mbox{-}B]\end{math} by $X^-$ and the rest of $X$ by \begin{math}X^|=X\setminus X[A\mbox{-}B]\end{math}.
    Note that $X^|$ need not be connected at all. 
    If \begin{math}X<Y<Z\end{math}, then \begin{math}Y^|\cp X^-=\emptyset\end{math} or \begin{math}Y^|\cp Z^-=\emptyset\end{math}. Indeed, otherwise \begin{math}Y\cp X=Y^|\cp X^-\end{math} and \begin{math}Y\cp Z=Y^|\cp Z^-\end{math} and then $X$ and $Z$ cannot intersect since $Y^-$ separates $X^|$ from $Z$ and separates $Z^|$ from $X$, a contradiction. See Figure \ref{fig:halfline}(b).
    Thus we can divide $\cP$ into two groups: 
    $$
    \cP^\perp=\{Y\in\cP: \textit{for }\forall X\in \cP\textit{ such that } X<Y
    \textit{ we have } X^-\cp Y^|=\emptyset\},
    $$
    \begin{displaymath}\cP^\top=\{Y\in \cP: \textit{for }  \forall Z\in \cP\textit{ such that } Z>Y\textit{ we have } Z^-\cp Y^|=\emptyset\}.\end{displaymath}
    Without loss of generality, assume that $|\cP^\top|\ge |\cP|/2$.
    
    Let the lowest member of $\cP^\top$ be $X_0$, the second lowest be $X_1$ and  \begin{math}\cP_0^\top= \cP^\top\setminus\{X_0\}\end{math}. As \begin{math}X_1^-\cp X_0^|=\emptyset\end{math}, $X_0$ remains completely below $X_1$ in the strip and thus for the rest of the trees in \begin{math}\cP_0^\top\end{math} to intersect $X_0$, they all have to go below $X_1^-$, i.e., for every \begin{math}X\in \cP_0^\top\end{math}, \begin{math}X\ne X_1\end{math}, we have that $X^|$ intersects $X_1^-$. Moreover, for \begin{math}X,Y\in \cP_0^\top\end{math}, \begin{math}X^-\cp Y^|\end{math} if and only if $X<Y$. For each \begin{math}X\in \cP_0^\top\end{math}, delete from it all parts that are not in $X^-$ or in the shortest path connecting \begin{math}X^|\cp X_0\end{math} with $X^-$. Notice that in this drawing for every \begin{math}X\in \cP_0^\top\end{math}, $X^|$ is a curve with one endpoint on $X^-$ and its part below $X_1^-$ is disjoint from all other trees in \begin{math}\cP_0^\top\end{math} and ends in \begin{math}X\cp X_0\end{math}. Thus $X_0$ can be easily redrawn to be a curve going close to the original tree, connecting \begin{math}X_0 \cp A\end{math} with \begin{math}X_0 \cp B\end{math} and touching each \begin{math}X\in \cP_0^\top\end{math} in the same order (and in the same point) as before. See Figure \ref{fig:halfline}(c) for the drawing at this stage, where the original $X_0$ is shown dotted.
    
    Now the trees in \begin{math}\cP^\top\end{math} can be easily redrawn as T-shapes, where $X_0$ becomes a horizontal segment and for an \begin{math}X\in \cP_0^\top\end{math} the hat of the respective T-shape is $X^-$, while the leg is a redrawing of the curve $X^|$ as a vertical segment. 
\end{proof}

\begin{lem}\label{lem:alphaordertype}
    For any finite point set $\cP$ and for every $\alpha>0$ there is a point set $\cP_\alpha$ such that for every $\cQ\subset \cP_\alpha$ for which \begin{math}|\cQ|\ge \alpha |\cP_\alpha|\end{math}, there is a \begin{math}\cQ'\subset \cQ\end{math} such that \begin{math}\o(\cQ')=\o(\cP)\end{math}, i.e., $\cP$ and $\cQ'$ have the same order type.
\end{lem}

\begin{cor}\label{cor:alphaordertype}    
    Let \cXTO be any family of total 3-orders (i.e., a subfamily of \cTO) such that \begin{math}\cXTO\subsetneq \cpTO\end{math}, then for every \begin{math}\alpha>0\end{math} there is a point set $\cP_\alpha$ such that for every \begin{math}\cQ\subset \cP_\alpha\end{math} for which \begin{math}|\cQ|\ge \alpha |\cP_\alpha|\end{math}, we have that \begin{math}\o(\cQ)\notin\cXTO\end{math}.
\end{cor}

\begin{proof}[of Lemma \ref{lem:alphaordertype}]
    We can assume \begin{math}\alpha<1\end{math}, since the statement holds for \begin{math}\alpha=1\end{math} by setting \begin{math}\cP_\alpha=\cP\end{math} and is vacuous for \begin{math}\alpha>1\end{math}.
    Let \begin{math}n=|\cP|\end{math} and \begin{math}\cP_1=\cP\end{math}.
    Let $\cP_{i+1}$ be the point set obtained from $\cP$ by replacing each \begin{math}p\in\cP\end{math} with a small copy of $\cP_i$ close to $p$, such that for any \begin{math}p,q,r\in \cP\end{math} and points \begin{math}p',q',r'\in \cP_{i+1}\end{math} where \begin{math}p',r',q'\end{math} is close to \begin{math}p,q,r\end{math}, respectively, we have \begin{math}\o(pqr)=\o(p'q'r')\end{math}. Starting with $\cP_1$ we repeat this process \begin{math}k-1\end{math} times, where $k$ is a large enough integer specified later, to get $\cP_k$ with \begin{math}N=n^k\end{math} points. 
    
    Notice that for every \begin{math}2\le i\le k\end{math}, if some subset of $\cP_{i}$ contains a point from each of the $n$ groups of $\cP_{i-1}$'s in $\cP_{i}$, then it contains \begin{math}\o(\cP)\end{math}. 
    
    Assume that $\cQ$ is a subset of $\cP$ that does not contain a subset with order type $\o(\cP)$. We need to show that $\cQ$ must be smaller than \begin{math}\alpha|\cP|\end{math}.
    
    First, by the previous observation $\cQ$ must completely avoid one of the groups of $\cP_{k-1}$ in $\cP_k$.
    Then the rest of the groups are homothetic to $\cP_{k-1}$, and thus in each of them $\cQ$ must completely avoid one of the groups of $\cP_{k-2}$. By a repeated application of this argument, in each step we find that $\cQ$ avoids \begin{math}\frac{1}{n}\end{math}th of the remaining points and thus altogether $\cQ$ can have at most \begin{math}(\frac{n-1}{n})^k\cdot N\end{math} points. Thus \begin{math}\frac{|\cQ|}{|\cP|}\le (\frac{n-1}{n})^k=(1-\frac 1n)^k<e^{-k/n}<\alpha\end{math} if $k$ is large enough, we are done.  
\end{proof}

\begin{rem}
Lemma \ref{lem:alphaordertype} also follows from the powerful Multidimensional Szemer\'edi theorem of \cite{FK}.
Indeed, we can pick $\cP_\alpha$ to be a very large grid, and any $\alpha$-dense subset of $\cP_\alpha$ will contain a grid of any given size (fixed before choosing the size of the large grid), that in turn will contain every order type of points up to some given size.
In fact, the copy $\cQ'$ of $\cP$ can even be chosen to be homothetic to $\cP$ if $\cP$ was a subset of some grid.
However, the size of $\cP_\alpha$ obtained this way will be enormous.
The above is a self-contained proof that gives a much better bound on $|\cP_\alpha|$, and a $\cQ'$ that is almost homothetic to $\cP$ in the sense that it can be the homothet of some arbitrarily slightly perturbed copy of $\cP$ (the perturbation cannot be fixed but the extent of the perturbation can).
\end{rem}

\begin{prop}\label{prop:Y}
    There exists a total 3-order realizable by points which is not realizable by Y-shapes, i.e., \begin{math}\cpTO\not\subset \cYTO\end{math}.
\end{prop}
\begin{proof}
    Fix $\alpha=1/72$.
    The construction $\cP$ will consist of a $\cP_\alpha$ such that for every \begin{math}\cQ\subset \cP_\alpha\end{math} for which \begin{math}|\cQ|\ge \alpha |\cP_\alpha|\end{math}, we have \begin{math}\o(\cQ)\not\in \cTTO\end{math}, and three more points, \begin{math}A,B,C\end{math}, such that \begin{math}\cP_\alpha\subset conv(ABC)\end{math}. This exists by Corollary \ref{cor:Tnotinp} and Corollary \ref{cor:alphaordertype}.
    We can assume by applying a suitable affine transformation (that leaves $\o(\cP_\alpha)$ the same) that $\cP_\alpha$ is flattened so that for any two points \begin{math}X,Y\in\cP_\alpha\end{math} the points $A$ and $B$ lay on different sides of the $XY$ line, i.e., \begin{math}\o(XYA)\ne\o(XYB)\end{math}.
    
    Assume on the contrary that \begin{math}\o(\cP)\end{math} has a realization with Y-shapes.
    By \ref{item:5points} every intersection occurs inside the closure of $\hollow(ABC)$, so we can assume that all Y-shapes are in the closure of $\hollow(ABC)$.
    As each of $A,B,C$ are Y-shapes, their union consists of \begin{math}\partial\hollow(ABC)\end{math} and without loss of generality possibly only some inward growing hairs.
    We cover $A$ by at most three simple curves, called sections (which can overlap): for each endpoint of the Y-shape $A$ the curve that starts at $A\cp B$ and ends in this endpoint is a section. We similarly define at most three sections that cover $B$.    
    Any pair of sections, one of $A$ and the other of $B$, can be simultaneously extended into topological halflines outside \begin{math}\hollow(ABC)\end{math}, such that they are pairwise disjoint, by shortening them slightly at \begin{math}A\cp B\end{math} and extending them at their other endpoints.
    As no Y-shape from $\cP_\alpha$ can contain $A\cp B$, the intersection structure among the Y-shapes and \begin{math}A,B\end{math} remains the same.
    
    Now we partition $\cP_\alpha$ into at most 36 groups depending on which of the three sections of $A$ and $B$ they intersect (\begin{math}3\cdot 3\end{math} options), and from which directions (going along a subcurve of $A$ from its branching point, on which side does the Y-shape touch it and similarly for $B$; \begin{math}2\cdot 2\end{math} options).
    By a slight redrawing, we can assume that no branching point of a Y-shape from $\cP_\alpha$ falls on $A$ or $B$.
    Recall that due to the flattening of $\cP_\alpha$, we have \begin{math}\o(XYA)\ne\o(XYB)\end{math} for every two different \begin{math}X,Y\notin \{A,B\}\end{math}, thus the elements of any group satisfy the conditions of
    Proposition \ref{prop:halfline}. 
    Fix a group that contains at least \begin{math}|\cP_\alpha|/36\end{math} elements, and apply Lemma \ref{lem:halfline} to redraw half of its elements (thus at least \begin{math}|\cP_\alpha|/72\end{math} elements) as T-shapes. On the other hand, by definition of $\cP_\alpha$, no $1/72$ fraction of \begin{math}|\cP_\alpha|\end{math} defines a total 3-order which is a \TTO, a contradiction.
\end{proof}

We note that in the proof of Proposition \ref{prop:Y} we used only that $A$ and $B$ are realized as $Y$-shapes, the rest of the topological trees could be arbitrary.

Using Proposition \ref{prop:Y}, we can prove the following.

\begin{prop}
    There exists a total 3-order realizable by points which is not realizable by topological trees, i.e., $\cpTO\not\subset \cTrTO$.
\end{prop}
\begin{proof}
    We will prove the following.
    If \begin{math}\cP_0=\{p_1,\dots,p_n\}\end{math} is a point set such that \begin{math}\o(\cP_0)\notin \cYTO\end{math}, then there is a point set $\cP_n$ such that \begin{math}\o(\cP_n)\notin \cTrTO\end{math}. Fix first a sequence of numbers \begin{math}N_1,N_2,\dots\end{math} which grows very fast but is otherwise arbitrary. Now denote by $\cP_i$ the point set obtained from $\cP_0$ by  placing $N_j$ points close to the point $p_j$ for all \begin{math}1\le j\le i\end{math} (the points are placed arbitrarily in the close vicinity of $p_j$). In particular, $\cP_n$ has \begin{math}\sum_{i=1}^{n} N_i\end{math} points.
    Our strategy is to prove for each $i$, going from $n$ to $1$, that if \begin{math}\o(\cP_i)\end{math} has a \TrTO representation where the points \begin{math}\{p_{i+1},\dots,p_n\}\end{math} are represented by crossing Y-shapes, then \begin{math}\o(\cP_{i-1})\end{math} has a \TrTO representation where the points \begin{math}\{p_i,\dots,p_n\}\end{math} are represented by crossing Y-shapes.
    From this, it follows that \begin{math}\o(\cP_n)\in \cTrTO\end{math} would imply \begin{math}\o(\cP_0)\in \cYTO\end{math}, contradicting our assumption.
    
    So fix $i$ and let us assume that \begin{math}\o(\cP_i)\end{math} has a \TrTO representation where the points \begin{math}\{p_{i+1},\dots,p_n\}\end{math} are represented by crossing Y-shapes.
    Denote the collection of the $N_i$ trees that correspond to points that are in the vicinity of $p_i$ by $T_i$, and the remaining \begin{math}N_0:=N_1+\dots+N_{i-1}+n-i\end{math} trees that correspond to points that are \emph{not} in the vicinity of $p_i$ by $T_0$.
  
    Now we want to show that we can redraw at least one tree $X$ from $T_i$ as a $Y$-shape, while leaving all the trees from $T_0$ unchanged, so that the set of their induced 3-order, \begin{math}\o(T_0\cup \{X\})\end{math}, remains the same.
    Note that this system does not include any triples that contain at least two trees from $T_i$.
    This means that we can ignore the intersection points between two trees from $T_i$.
    Therefore, we can assume that every leaf of a tree $X$ from $T_i$ is an intersection point with a tree from $T_0$, by pruning parts of $X$, if necessary. While members of $T_i$ might become disjoint, it still holds that the members of \begin{math}T_0\cup \{X\}\end{math} are pairwise intersecting for any \begin{math}X\in T_i\end{math}.
    
    If every tree from $T_i$ has at least four leaves, intersecting a tree of $T_0$ in each of them, then we fix four such leaves for each $T_i$.
    Divide $T_i$ into \begin{math}\binom{N_0}4\end{math} groups depending on which four trees they intersect in their four fixed leaves.
    By Lemma \ref{lem:treetangencies}, no group can have more than $c$ trees (where $c$ is as in the lemma). 
    Thus, if we choose \begin{math}N_i\ge c\binom{N_0}4\end{math}, then we get a contradiction.
    
    Thus, there is a tree from $T_i$ which has at most three leaves, this is a redrawing of the original $T_i$ as a Y-shape crossing every tree from $T_0$, exactly as we wanted.
\end{proof}

As \cTrTO= \cGCTO by Corollary \ref{cor:GCTr}, this finishes the proof of Theorem \ref{thm:main}.

\section{Open problems}\label{sec:discussion}

\begin{prob}
    Determine the smallest number of points which are not representable as a \GCTO, as a \CTO, etc.
\end{prob}

\begin{prob}
    Determine the containment relations among the subclasses of \cTO. Is our diagram in Figure \ref{fig:3TOposet} complete, or are there more containments?
\end{prob}

\begin{prob}[by an anonymous referee.]
    Is there a nice topological representation for every \cPO/\cTO?
\end{prob}

\begin{prob}[by an anonymous referee.]
    How do our subclasses of \cTO compare to the class given by Topological drawings of the complete graph, introduced by \cite{BFSchSchS}?
\end{prob}

\begin{prob}
    Determine the number/growth rate of \cGCTO, \cCTO, etc., on $n$ elements.
\end{prob}

Finally, we would like to pose the following strengthening of the thrackle conjecture.

\begin{conj}
    Suppose that we are given in the plane $n$ points, $\cP$, and $m$ topological trees that pairwise intersect exactly once such that each leaf of each tree is from $\cP$.
    We conjecture that $m\le n$.
\end{conj}

Note that if each tree consists of a single edge, then we get back the original thrackle conjecture. We could weaken our conjecture by requiring that the branching points of the trees also need to be from $\cP$. Another interesting, and probably easier, special case is when we also require that each edge of each tree needs to be a segment between two points of $\cP$.

\acknowledgements

We would like to thank N\'ora Frankl and M\'arton Nasz\'odi for discussions during the early stage of the project and a reviewer for their useful suggestions.

\nocite{*}
\bibliographystyle{abbrvnat}
\bibliography{gc_dmtcs_final}
\label{sec:biblio}

\end{document}